\documentclass[12pt]{article}

\usepackage{times}

\usepackage{enumitem}
\usepackage[utf8]{inputenc}
\usepackage{commath}
\usepackage{amsthm, amscd}
\usepackage{amsmath}
\usepackage{amssymb}
\usepackage{cite}
\usepackage{setspace}
\usepackage{ stmaryrd }
\usepackage{tikz-cd}

\usepackage{tocloft}
\usepackage{sectsty}
\usepackage{xparse}

\newcommand*\tasklabelformat[1]{#1)}

\usepackage{titlesec}

\usepackage{tasks}[newest]
\settasks{
  label = \alph* ,
  label-format = \tasklabelformat ,
  label-width  = 12pt
}

\usepackage{bigints}
\usepackage{comment}
\usepackage{mathtools}
\usepackage{mathrsfs}
\usepackage{fancyhdr}

\allowdisplaybreaks[4]

\numberwithin{equation}{section}
\usepackage{etoolbox}
\patchcmd{\thebibliography}
  {\settowidth}
  {\setlength{\itemsep}{0pt plus -10pt}\settowidth}
  {}{}
\apptocmd{\thebibliography}
  {
  }
  {}{}
\makeatletter
\newtheorem*{rep@theorem}{\rep@title}
\newcommand{\newreptheorem}[2]{%
\newenvironment{rep#1}[1]{%
 \def\rep@title{#2 \ref{##1}}%
 \begin{rep@theorem}}%
 {\end{rep@theorem}}}
\makeatother

\theoremstyle{theorem}

\newreptheorem{theorem}{Theorem}
\newtheorem{thm}{Theorem}[section]
\newtheorem*{thm*}{Theorem}
\theoremstyle{definition}
\newtheorem{prop}[thm]{Proposition}
\newtheorem*{prop*}{Proposition}
\newtheorem{defn}[thm]{Definition}
\newtheorem{lem}[thm]{Lemma}
\newtheorem{cor}[thm]{Corollary}
\newtheorem*{cor*}{Corollary}
\theoremstyle{remark}

\newtheorem{rem}[thm]{Remark}

\title{\vspace*{-1.5cm} Toeplitz operators, submultiplicative filtrations
\\
and weighted Bergman kernels}

\author
{Siarhei Finski
}

\date{}

\usepackage[%
    left=1in,%
    right=1in,%
    top=1.1in,%
    bottom=0.8in,%
    paperheight=11in,%
    paperwidth=8.5in%
]{geometry}

\newcommand{\imun} {\sqrt{-1}}

\newcommand{\comp}{\mathbb{C}}
\newcommand{\real}{\mathbb{R}}

\newcommand{\nat}{\mathbb{N}}
\newcommand{\integ}{\mathbb{Z}}

\newcommand{\enmr}[1]{\text{End}{(#1)}}

\newcommand{\ccal}{\mathscr{C}}

\newcommand{\dbar}{ \overline{\partial} }

\newcommand{\tr}[1]{{\rm{Tr}} \big[ #1 \big]}

\makeatletter
\DeclareFontFamily{OMX}{MnSymbolE}{}
\DeclareSymbolFont{MnLargeSymbols}{OMX}{MnSymbolE}{m}{n}
\SetSymbolFont{MnLargeSymbols}{bold}{OMX}{MnSymbolE}{b}{n}
\DeclareFontShape{OMX}{MnSymbolE}{m}{n}{
    <-6>  MnSymbolE5
   <6-7>  MnSymbolE6
   <7-8>  MnSymbolE7
   <8-9>  MnSymbolE8
   <9-10> MnSymbolE9
  <10-12> MnSymbolE10
  <12->   MnSymbolE12
}{}
\DeclareFontShape{OMX}{MnSymbolE}{b}{n}{
    <-6>  MnSymbolE-Bold5
   <6-7>  MnSymbolE-Bold6
   <7-8>  MnSymbolE-Bold7
   <8-9>  MnSymbolE-Bold8
   <9-10> MnSymbolE-Bold9
  <10-12> MnSymbolE-Bold10
  <12->   MnSymbolE-Bold12
}{}

\let\llangle\@undefined
\let\rrangle\@undefined
\DeclareMathDelimiter{\llangle}{\mathopen}%
                     {MnLargeSymbols}{'164}{MnLargeSymbols}{'164}
\DeclareMathDelimiter{\rrangle}{\mathclose}%
                     {MnLargeSymbols}{'171}{MnLargeSymbols}{'171}
\makeatother

\newenvironment{sciabstract}{}

\cftsetindents{section}{0em}{1.7em}

\sectionfont{\large}

\titlespacing*{\section}
{0pt}{5pt}{5pt}

\begin{document}

\maketitle 

\vspace*{-0.7cm}

\vspace*{0.3cm}

\begin{sciabstract}
  \textbf{Abstract.}
 We demonstrate that the weight operator associated with a submultiplicative filtration on the section ring of a polarized complex projective manifold is a Toeplitz operator.
	We further analyze the asymptotics of the associated weighted Bergman kernel, presenting the local refinement of earlier results on the convergence of jumping measures for submultiplicative filtrations towards the pushforward measure defined by the corresponding geodesic ray.
\end{sciabstract}

\pagestyle{fancy}
\lhead{}
\chead{Toeplitz operators, submultiplicative filtrations and weighted Bergman kernels}
\rhead{\thepage}
\cfoot{}


\newcommand{\Addresses}{{
  \bigskip
  \footnotesize
  \noindent \textsc{Siarhei Finski, CNRS-CMLS, École Polytechnique F-91128 Palaiseau Cedex, France.}\par\nopagebreak
  \noindent  \textit{E-mail }: \texttt{finski.siarhei@gmail.com}.
}} 

\vspace*{0.25cm}

\par\noindent\rule{1.25em}{0.4pt} \textbf{Table of contents} \hrulefill

\vspace*{-1.5cm}

\tableofcontents

\vspace*{-0.2cm}

\noindent \hrulefill


\section{Introduction}\label{sect_intro}
	The main goal of this paper is to make a connection between the theory of Toeplitz operators, Bergman kernels and submultiplicative filtrations.
	\par 
	Throughout the whole article, we fix a complex projective manifold $X$, $\dim X = n$, and an ample line bundle $L$ over it. 
	Recall that a decreasing graded filtration $\mathbb{R} \ni \lambda \mapsto \mathcal{F}^{\lambda} R(X, L) := \oplus_{k = 0}^{+\infty} \mathcal{F}_k^{\lambda} H^0(X, L^{\otimes k})$, $\mathcal{F}_k^{\lambda} H^0(X, L^{\otimes k}) \subset H^0(X, L^{\otimes k})$, on the section ring, 
	\begin{equation}
		R(X, L) := \oplus_{k = 0}^{+\infty} H^0(X, L^{\otimes k}),
	\end{equation}
	is called \textit{submultiplicative} if the multiplication map on $R(X, L)$ factors through 
	\begin{equation}\label{eq_sm_filtr_defn}
		\mathcal{F}_k^{\lambda} H^0(X, L^{\otimes k}) \otimes \mathcal{F}_l^{\mu} H^0(X, L^{\otimes l}) \to \mathcal{F}_{k + l}^{\lambda + \mu} H^0(X, L^{\otimes (k + l)}), \quad \text{ for any } \lambda, \mu \in \mathbb{R}, k, l \in \mathbb{N}.
	\end{equation}
	\par The most geometrically natural example is the filtration given by the order of vanishing along a fixed divisor. Other examples include filtrations associated with the weight of a $\mathbb{C}^*$-action on the pair $(X, L)$, filtrations given by the restriction of the Harder-Narasimhan filtration on direct image vector bundles associated with tensor powers of a polarization on a family of manifolds \cite{ChenHNolyg}, \cite{FinHNI}, filtrations associated with valuations or graded ideals \cite{ReesBook}, or finitely generated filtrations induced by an arbitrary filtration on $H^0(X, L^{\otimes k})$, for $k \in \mathbb{N}$ big enough.
	\par 
	Remark that any (decreasing) filtration $\mathcal{F}$ on the Hermitian vector space $(V, H)$ induces the \textit{weight operator} $A(\mathcal{F}, H) \in {\textrm{End}}(V)$, defined as
	\begin{equation}
		A(\mathcal{F}, H) e_i = w_{\mathcal{F}}(e_i) \cdot e_i, \quad \text{where} \quad w_{\mathcal{F}}(e) := \sup \{ \lambda \in \real : e \in \mathcal{F}^{\lambda} V \}, e \in V,
	\end{equation}
	and $e_1, \ldots, e_{r}$, $r := \dim V$, is an orthonormal basis of $(V, H)$ adapted to the filtration $\mathcal{F}$ in the sense that $e_1$ has the maximal weight, $e_2$ has the maximal weight among vectors orthogonal to $e_1$, and so on.
	We call $w_{\mathcal{F}} : V \to ]- \infty, +\infty]$ the weight function of $\mathcal{F}$.
	The main purpose of this article is to study the weight operator associated with submultiplicative filtrations.
	\par 
	\begin{sloppypar}
	More precisely, we fix a positive Hermitian metric $h^L$ on $L$, and denote by ${\textrm{Hilb}}_k(h^L)$ the $L^2$-metric on $H^0(X, L^{\otimes k})$ induced by $h^L$, see (\ref{eq_defn_l2}).
	For a continuous function $g : \real \to \real$, we consider the \textit{weighted Bergman kernel}, $B_k^{\mathcal{F}, g}(x) \in \real$, $k \in \mathbb{N}$, $x \in X$, defined as 
	\begin{equation}\label{eq_weight_berg}
		B_k^{\mathcal{F}, g}(x) := \sum_{i = 1}^{N_k} \Big\langle g \Big( \frac{A(\mathcal{F}_k, {\textrm{Hilb}}_k(h^L))}{k} \Big) s_{i, k}(x), s_{i, k}(x) \Big\rangle_{h^{L^{\otimes k}}},
	\end{equation}
	where $N_k := \dim H^0(X, L^{\otimes k})$ and $s_{i, k}$, $i = 1, \ldots, N_k$, is an orthonormal basis of $(H^0(X, L^{\otimes k}), {\textrm{Hilb}}_k(h^L))$.
	The reader will check that (\ref{eq_weight_berg}) doesn't depend on the choice of an orthonormal basis, and if the basis is adapted to $\mathcal{F}_k$, then $B_k^{\mathcal{F}, g}(x) = \sum_{i = 1}^{N_k} g(w_{\mathcal{F}_k}(s_{i, k}) / k) |s_{i, k}(x)|_{h^{L^{\otimes k}}}^2$.
	When $\mathcal{F}$ coincides with the grading on $R(X, L)$ (we then say $\mathcal{F}$ is trivial), $B_k^{\mathcal{F}, g}$ coincides with the usual Bergman kernel up to a constant.
	\end{sloppypar}
	\par 
	Our first goal is to study the behavior of $B_k^{\mathcal{F}, g}$, as $k \to \infty$.
	As we shall see, this study depends largely on the algebraic properties of $\mathcal{F}$.
	\par 
	We say that a submultiplicative filtration $\mathcal{F}$ on $R(X, L)$ is \textit{bounded} if there is $C > 0$, such that for any $k \in \nat^*$, $\mathcal{F}^{ C k} H^0(X, L^{\otimes k}) = \{0\}$.
	We say that $\mathcal{F}$ is \textit{finitely generated} if it has integral weights and the associated $\comp[\tau]$-algebra ${\rm{Rees}}(\mathcal{F}) := \sum_{(\lambda, k) \in \mathbb{Z} \times \mathbb{N}} \tau^{- \lambda} \mathcal{F}^{\lambda} H^0(X, L^{\otimes k})$, also called the \textit{Rees algebra}, is finitely generated.
	Finitely generated submultiplicative filtrations are clearly automatically bounded.
	As the section ring $R(X, L)$ is finitely generated, cf. \cite[Example 2.1.30]{LazarBookI}, the set of finitely generated submultiplicative filtrations is non-empty, and for an arbitrary submultiplicative filtration, there is $C > 0$, such that for any $k \in \nat^*$, $\mathcal{F}^{- C k} H^0(X, L^{\otimes k}) = H^0(X, L^{\otimes k})$.
	\par 
	Recall that Phong-Sturm \cite[Theorem 3]{PhongSturmDirMA} and Ross-Witt Nystr{\"o}m \cite{RossNystAnalTConf} associated for an arbitrary bounded submultiplicative filtration on $R(X, L)$ the \textit{geodesic ray} $h^{\mathcal{F}}_t$, $t \in [0, +\infty[$, of Hermitian metrics on $L$, emanating from $h^L$.
	In general $h^{\mathcal{F}}_t$ is not smooth, however, due to convexity properties in $t$-variable, one can always define its derivative at $t = 0$, $\dot{h}^{\mathcal{F}}_0 := (h^{\mathcal{F}}_0)^{-1} \frac{d}{dt} h^{\mathcal{F}}_t|_{t = 0}: X \to \real$, and this derivative is bounded, see Section \ref{sect_prel} for details.
	We denote $\phi(h^L, \mathcal{F}) = - \dot{h}^{\mathcal{F}}_0$ for brevity.
	\par 
	\begin{thm}\label{thm_berg_conv}
		For a bounded submultiplicative filtration $\mathcal{F}$ on $R(X, L)$, the sequence of functions $x \mapsto \frac{1}{k^n} B_k^{\mathcal{F}, g}(x)$, $x \in X$, $k \in \nat$, is uniformly bounded and converges pointwise to a function which equals $g(\phi(h^L, \mathcal{F}))$ almost everywhere.
		If, moreover, $\mathcal{F}$ is finitely generated, then the limit coincides with $g(\phi(h^L, \mathcal{F}))$ everywhere, and the convergence is uniform.
	\end{thm}
	\begin{rem}\label{rem_berg_conv}
		a) When $\mathcal{F}$ is trivial, Theorem \ref{thm_berg_conv} recovers the result of Tian \cite{TianBerg} concerning the Bergman kernel asymptotics, and $\phi(h^L, \mathcal{F})$ equals to $1$ (we emphasize that our proof builds on Tian's result!).
		When the filtration is induced by the $\mathbb{C}^*$-action on $(X, L)$, so that the induced $S^1$-action is isometric for the Kähler form associated with $h^L$, the result is also well known, \cite[\S 7.3]{SzekBook}, \cite{MaZhBKSR}, and $\phi(h^L, \mathcal{F})$ then coincides with the Hamiltonian of the induced $S^1$-action. 
		See also \cite{OodOffDiag} for a recent result on the off-diagonal behavior in the related context.
		\par 
		b) In the end of Section \ref{sect_prel}, we provide several examples indicating that Theorem \ref{thm_berg_conv} is sharp. 
		Specifically, we show that without the finite-generation assumption, pointwise convergence cannot be strengthened to uniform convergence, and the limit doesn't coincide with $g(\phi(h^L, \mathcal{F}))$ everywhere. 
		We then show that with finite-generation assumption, uniform convergence cannot be improved to $\mathscr{C}^1$-convergence.
		\par 
		c) We explain in (\ref{eq_der_zeld_song}) that Theorem \ref{thm_berg_conv} partially justifies a folklore conjecture on the $\mathscr{C}^1$-convergence of quantized geodesic rays towards the geodesic ray associated with the filtration.
	\end{rem}
	\par 
	\begin{sloppypar}
	We will now describe the context behind Theorem \ref{thm_berg_conv}.
	Remark the following basic identity 
	\begin{equation}\label{eq_basic_id}
		\frac{1}{n!} \int B_k^{\mathcal{F}, g}(x) c_1(L, h^L)^n = {\rm{Tr}} \Big[ g \Big( \frac{A(\mathcal{F}_k, {\textrm{Hilb}}_k(h^L))}{k} \Big) \Big],
	\end{equation}
	which in particular implies that Theorem \ref{thm_berg_conv} guarantees the weak convergence of the spectral measures of the operators $\frac{1}{k} A(\mathcal{F}_k, {\textrm{Hilb}}_k(h^L))$ towards the pushforward measure $\phi(h^L, \mathcal{F})_* (c_1(L, h^L)^n / \int c_1(L)^n)$, as $k \to \infty$. 
	From the definitions, we see that these spectral measures coincide (up to a normalization) with the jumping measures associated with the filtration, cf.  \cite{NystOkounTest}, \cite{HisamSpecMeas}, where the latter probability measure on $\real$ is defined as 
	\begin{equation}\label{eq_jump_meas_d}
		\mu_{\mathcal{F}, k} := \frac{1}{N_k} \sum_{j = 1}^{N_k} \delta \Big[ \frac{e_{\mathcal{F}}(j, k)}{k} \Big], 
	\end{equation}
	where $\delta[x]$ is the Dirac mass at $x \in \real$ and $e_{\mathcal{F}}(j, k)$ are the \textit{jumping numbers}, defined as follows
	\begin{equation}\label{eq_defn_jump_numb}
		e_{\mathcal{F}}(j, k) := \sup \Big\{ t \in \real : \dim \mathcal{F}^t H^0(X, L^{\otimes k}) \geq j \Big\}.
	\end{equation}
	\par 
	Weak convergence of jumping measures was first established by Chen \cite{ChenHNolyg} and Boucksom-Chen \cite{BouckChen}. 
	Subsequently, Witt Nyström \cite{NystOkounTest} proved that for filtrations associated with a $\mathbb{C}^*$-action, the weak limit coincides with the pushforward measure; he also conjectured the analogous relation for finitely generated filtrations. 
	Hisamoto established this in \cite{HisamSpecMeas}, and the author \cite[Theorem 5.4]{FinSecRing} further extended it for bounded submultiplicative filtrations.
	In light of (\ref{eq_basic_id}), Theorem \ref{thm_berg_conv} can be interpreted as a local refinement of these statements, showing that \textit{the convergence occurs at the level of functions themselves}, rather than solely at the level of their integrals.
	\end{sloppypar}
	\par 
	We invite the reader to compare Theorem \ref{thm_berg_conv} with Berman-Boucksom-Witt Nyström \cite[Theorem B]{BerBoucNys} and Darvas-Xia \cite[Theorem 1.2]{DarvXiaVolumes}, where authors establish the convergence in weak topology of partial Begman kernels associated with Nadel multiplier ideal sheaves of plurisubharmonic potentials.
	Remark however that filtrations associated with multiplier ideals are not necessarily submultiplicative (because a product of two $L^2$-integrable sections is not necessarily $L^2$-integrable), and so it seems that there is no direct connection between our findings.
	\par 
	Theorem \ref{thm_berg_conv} in particular applies to the filtration associated with the vanishing order along a submanifold. 
	Related study on the partial Bergman kernel was initiated by Berman \cite[Theorem 4.3]{BerPartBerg}, and then developed by Ross-Singer \cite{RossSinger}, Coman-Marinescu \cite{ComMarPartBerg}, Zelditch-Zhou \cite{ZeldZhouInter}, Sun \cite{SunJingPartDens}, and others. 
	As partial Bergman kernel corresponds to indicator functions $g$ in (\ref{eq_weight_berg}), which are not continuous, our study doesn't apply directly to the partial Bergman kernel.
	\par 
	Theorem \ref{thm_berg_conv} is related with a much more refined result concerning the asymptotic properties of the weight operator. 
	To explain this statement, let us recall a version of \cite[Definition 7.2.1]{MaHol}. 
	\begin{defn}\label{defn_toepl}
		A sequence of operators $T_k \in {\enmr{H^0(X, L^{\otimes k})}}$, $k \in \nat$, is called a \textit{Toeplitz operator} if there is a continuous function $f: X \to \real$, called the symbol of $\{T_k\}_{k = 0}^{+ \infty}$, such that for any $\epsilon > 0$, there is $k_0 \in \nat$, such that for every $k \geq k_0$, we have
		\begin{equation}
			\big \| T_k -  T_{f, k} \big \| \leq \epsilon,
		\end{equation}
		where $\| \cdot \|$ is the operator norm on ${\enmr{H^0(X, L^{\otimes k})}}$, subordinate to ${\textrm{Hilb}}_k(h^L)$, $T_{f, k} := B_k \circ M_{f, k}$, and $B_k : L^{\infty}(X, L^{\otimes k}) \to H^0(X, L^{\otimes k})$ is the orthogonal (Bergman) projection to $H^0(X, L^{\otimes k})$, and $M_{f, k} : H^0(X, L^{\otimes k}) \to L^{\infty}(X, L^{\otimes k})$ is the multiplication map by $f$, acting as $s \mapsto f \cdot s$.
	\end{defn}
	Our second main result goes as follows.
	\begin{thm}\label{thm_main1}
		For finitely generated submultiplicative filtrations $\mathcal{F}$, the rescaled weight operator, $\{ \frac{1}{k} A(\mathcal{F}_k, {\textrm{Hilb}}_k(h^L)) \}_{k = 0}^{+\infty}$, forms a Toeplitz operator with symbol  $\phi(h^L, \mathcal{F})$. 
	\end{thm}
	\par 
	It turns out that the assumption on finitely generatedness is crucial in Theorem \ref{thm_main1} and quite surprisingly, it cannot be replaced by the regularity assumption on the geodesic ray associated with the filtration, see the end of Section \ref{sect_prel} for an explicit example showing this.
	\par 
	To cover the case of a general bounded submultiplicative filtration, we introduce another definition which, though natural, appears not to have been previously considered.
	\begin{defn}\label{defn_toepl_sch}
		A sequence of operators $T_k \in {\enmr{H^0(X, L^{\otimes k})}}$ is called a \textit{Toeplitz operator of Schatten class} if it is uniformly bounded in operator's norm, i.e. there is $C > 0$, such that $\| T_k \| \leq C$, for any $k \in \nat$, and  there is $f \in L^{\infty}(X)$, called the symbol of $\{T_k\}_{k = 0}^{+ \infty}$, so that for any $\epsilon > 0$, $p \in [1, +\infty[$, there is $k_0 \in \nat$, such that for every $k \geq k_0$, in the notations of Definition \ref{defn_toepl},
		\begin{equation}\label{eq_toepl_schatten}
			\big \| T_k -  T_{f, k} \big \|_p \leq \epsilon,
		\end{equation}
		where $\| \cdot \|_p$ is the $p$-Schatten norm, defined for an operator $A \in {\enmr{V}}$, of a finitely-dimensional Hermitian vector space $(V, H)$ as $\| A \|_p = (\frac{1}{\dim V} {\rm{Tr}}[|A|^p])^{\frac{1}{p}}$, $|A| := (A A^*)^{\frac{1}{2}}$.
	\end{defn}
	\par 
	We can now state the final result of this paper.
	\begin{sloppypar}
	\begin{thm}\label{thm_main2}
		For any bounded submultiplicative filtration $\mathcal{F}$, the rescaled weight operator, $\{ \frac{1}{k} A(\mathcal{F}_k, {\textrm{Hilb}}_k(h^L)) \}_{k = 0}^{+\infty}$, forms a Toeplitz operator of Schatten class with symbol $\phi(h^L, \mathcal{F})$. 
	\end{thm}
	\end{sloppypar}
	\par 
	Let us describe the relation between Theorems \ref{thm_main1}, \ref{thm_main2} and Theorem \ref{thm_berg_conv}.
	Remark that it is well known that the space of Toeplitz operators forms an algebra, and the symbol map is an algebra morphism, cf. \cite{BordMeinSchli}, \cite[\S 7]{MaHol}. 
	The analogue of this statement extends to Toeplitz operators of Schatten class, see Proposition \ref{prop_prod_toepl}, and it implies that the vector space of Toeplitz operators of Schatten class is closed under the continuous functional calculus.
	As we shall explain in Section \ref{sect_func_calcul}, this result along with some basic properties of the diagonal kernel show that Theorem \ref{thm_main1} implies the second part of Theorem \ref{thm_berg_conv}, and Theorem \ref{thm_main2} implies the first part of Theorem \ref{thm_berg_conv}, if one replaces the pointwise convergence by the convergence in $L^p(X)$-spaces, for any $p \in [1, +\infty[$.
	However, the diagonal kernels of Toeplitz operators of Schatten class do not necessarily converge pointwise (just as convergence in $L^p(X)$ of a sequence of functions doesn't imply the pointwise convergence), and so Theorem \ref{thm_berg_conv} is not a consequence of Theorems \ref{thm_main1}, \ref{thm_main2}.
	\par
	This paper is organized as follows. 
	In Section \ref{sect_prel}, we recall the preliminaries.
	In Section \ref{sect_pf}, we prove Theorems \ref{thm_main1}, \ref{thm_main2} modulo a similar statement concerning the transfer operator between $L^2$-norms, which we establish in Section \ref{sect_transfer}.
	In Section \ref{sect_func_calcul}, we discuss the convergence properties of the diagonal kernels of Toeplitz operators and as a consequence establish Theorem \ref{thm_berg_conv} where the pointwise convergence is replaced by the convergence in $L^p(X)$-spaces, for any $p \in [1, +\infty[$.
	The pointwise convergence part from Theorem \ref{thm_berg_conv} is then established in Section \ref{sect_subadd}.
	\par 
	To conclude, we emphasize that we used the results of this paper in \cite{FinHYM} to construct approximate solutions to a certain equation arising in relation to the Wess-Zumino-Witten equation.

	\paragraph{Notation.}  
	A \textit{filtration} $\mathcal{F}$ of a vector space $V$ is a map from $\real$ to vector subspaces of $V$, $t \mapsto \mathcal{F}^t V$, verifying $\mathcal{F}^t V \subset \mathcal{F}^s V$ for $t > s$, and such that $\mathcal{F}^t V  = V$ for sufficiently small $t$ and $\mathcal{F}^t V = \{0\}$ for sufficiently big $t$.
	We assume that this map is left-continuous, i.e. for any $t \in \real$, there is $\epsilon_0 > 0$, such that $\mathcal{F}^t V = \mathcal{F}^{t - \epsilon} V $ for any $0 < \epsilon < \epsilon_0$.
	\par 
	For a given function $f$ on a topological space, we denote by $f^*$ (resp. $f_*$) the upper (resp. lower) semi-continuous regularization of $f$. The same notations are used for metrics on line bundles. 
	\par 
	By a positive Hermitian metric on a line bundle we mean a smooth Hermitian metric with strictly positive curvature.
	For a positive Hermitian metric $h^L$ on $L$ (resp. and a fixed Kähler form $\chi$), we denote by ${\textrm{Hilb}}_k(h^L)$ (resp. ${\textrm{Hilb}}_k(h^L, \chi)$) the Hermitian metric on $H^0(X, L^{\otimes k})$ defined for arbitrary $s_1, s_2 \in H^0(X, L^{\otimes k})$ as follows 
	\begin{equation}\label{eq_defn_l2}
		\begin{aligned}
			&
			\langle s_1, s_2 \rangle_{{\textrm{Hilb}}_k(h^L)} = \frac{1}{n!} \int_X \langle s_1(x), s_2(x) \rangle_{{h^{L^{\otimes k}}}} \cdot c_1(L, h^L)^n,
			\\
			&
			\langle s_1, s_2 \rangle_{{\textrm{Hilb}}_k(h^L, \chi)} = \frac{1}{n!} \int_X \langle s_1(x), s_2(x) \rangle_{{h^{L^{\otimes k}}}} \cdot \chi^n.
		\end{aligned}
	\end{equation}
	\par 
	Let $V$ be a complex vector space, $\dim V = n$, and $\mathcal{H}_V$ be the space of Hermitian norms on $V$. 
	For $H_0, H_1 \in \mathcal{H}_V$, the \textit{transfer map}, $T \in {\rm{End}}(V)$, between $H_0, H_1$, is the Hermitian operator (with respect to both $H_0, H_1$), defined so that the Hermitian products $\langle \cdot, \cdot \rangle_{H_0}, \langle \cdot, \cdot \rangle_{H_1}$ induced by $H_0$ and $H_1$, are related as $\langle \cdot, \cdot \rangle_{H_1} = \langle \exp(-T) \cdot, \cdot \rangle_{H_0}$.
	For Hermitian $A_0, A_1 \in {\rm{End}}(V)$ on a Hermitian vector space $(V, H)$, we say that $A_0 \geq A_1$ if the difference $A_0 - A_1$ is positively-definite.
	For a $\mathscr{C}^1$-path $H_t \in \mathcal{H}_V$, we denote by $\dot{H}_t := H_t^{-1} \frac{d}{dt} H_t$.
	\par 
	We denote by $\mathbb{D}(a, b)$ the complex annuli if inner and outer radiuses $a$ and $b$ respectively.
	We denote by $\mathbb{D}$ the complex unit disc.
	For a Kähler form $\omega$ on $X$, we denote by ${\rm{PSH}}(X, \omega)$ the space of $\omega$-quasipsh potentials, consisting of functions $\psi : X \to [- \infty, +\infty[$, which are locally the sum of a psh function and of a smooth function so that the $(1, 1)$-current $\omega + \imun \partial \dbar \psi$ is positive.
	We denote $\pi : X \times \mathbb{D}(e^{-1}, 1) \to X$ and $z : X \times \mathbb{D}(e^{-1}, 1) \to \mathbb{D}(e^{-1}, 1)$ the usual projections.
	\par 
	For $T_k \in {\rm{End}}(H^0(X, L^{\otimes k}))$ and $x \in X$, we denote 
	\begin{equation}
		T_k(x) := \sum_{i = 1}^{N_k} \big\langle T_k s_{i, k}(x), s_{i, k}(x) \big\rangle_{h^{L^{\otimes k}}},
	\end{equation}
	in the notations of (\ref{eq_weight_berg}), and call it the diagonal kernel of $T_k$.
	Recall that the diagonal Bergman kernel $B_k(x)$, $k \in \nat$, $x \in X$, defined for an orthonormal basis $s_{i, k}$, $i = 1, \ldots, N_k$, $N_k := \dim H^0(X, L^{\otimes k})$, of $(H^0(X, L^{\otimes k}), {\textrm{Hilb}}_k(h^L))$, as $B_k(x) = \sum_{i = 1}^{N_k}  |s_{i, k}(x)|_{h^{L^{\otimes k}}}^2$.
We also define $B_k(x, y) \in L_x^{\otimes k} \otimes (L_y^{\otimes k})^*$ as $B_k(x, y) = \sum_{i = 1}^{N_k}  s_{i, k}(x) \otimes s_{i, k}(y)^*$.
	
	\paragraph{Acknowledgement.} 
	Author acknowledges the support of CNRS and École Polytechnique. 
	
	\section{The spaces of norms, Kähler forms and geodesic rays}\label{sect_prel}
	The primary aim of this section is to revisit some well-known results that will be utilized in the following sections. 
	Specifically, we begin by discussing results related to the geometry of the space of Hermitian norms on a finite-dimensional vector space and complex interpolation. 
	We also recall the connection with the space of filtrations.
	\par 
	Next, we examine the Mabuchi geometry of the space of Kähler potentials and recall the correspondence between test configurations and submultiplicative filtrations. 
	This is followed by a construction of geodesic rays from test configurations and submultiplicative filtrations.
	\par 
	The section concludes with an explicit calculation of a geodesic ray associated with two filtrations, highlighting the sharpness of the paper’s main results.
	
	\par 
	\paragraph{Geodesics in the space of Hermitian norms.}
	The space $\mathcal{H}_V$ of  Hermitian norms on a complex vector space $V$, $\dim V = n$, carries a natural metric, cf. \cite[\S 6]{BhatiaBook}.
	A path $H_t \in \mathcal{H}_V$, $t \in [0, 1]$, is the \textit{geodesic} between $H_0, H_1 \in \mathcal{H}_V$ with respect to this metric, if for the transfer map, $T \in {\rm{End}}(V)$, between $H_0, H_1$, the endomorphism $t T$, $t \in [0, 1]$ is the transfer map between $H_0$ and $H_t$.
	\par 
	It is possible to view the space of filtrations on $V$ as the boundary at the infinity of $\mathcal{H}_V$, where the latter space is interpreted in terms of geodesic rays.
	For this, for any filtration $\mathcal{F}$ on $V$, we associate a geodesic ray in $\mathcal{H}_V$ as follows.
	We fix $H_V \in \mathcal{H}_V$, and define $H_t^{\mathcal{F}} \in \mathcal{H}_V$, $t \in [0, +\infty[$, through the associated Hermitian product as $\langle \exp(-t A(\mathcal{F}, H_V)) \cdot, \cdot \rangle_{H_V}$, where $A(\mathcal{F}, H_V) \in {\rm{End}}(V)$ is the weight operator associated with $\mathcal{F}$, defined in the Introduction.
	It is immediate that $H_t^{\mathcal{F}}$ is a geodesic ray emanating from $H_V$.
	\par
	It is well known that geodesics between Hermitian norms can be seen through the lens of complex interpolation theory, cf. \cite[Theorem 5.4.1]{InterpSp}.
	Let us recall the following crucial statement concerning the order-preserving properties of complex interpolation theory.
	\begin{prop}[{cf. \cite[Theorem 5.4.1]{InterpSp} and \cite[Proposition 4.12]{FinTits}}]\label{thm_interpol}
		For any $H_i \in \mathcal{H}_V$, $i = 0, 1, 2$, verifying $H_1 \leq H_2$, the geodesics $H_t^1$, $H_t^2$, $t \in [0, 1]$, between $H_0$, $H_1$ and $H_0$, $H_2$ respectively, compare as $H_t^1 \leq H_t^2$.
		In particular, we have $H_0^{-1} \dot{H}^1_0  \leq H_0^{-1} \dot{H}^2_0$.
		Similarly, for any filtrations $\mathcal{F}_i$, $i = 1, 2$, on $V$, so that their weight functions are related as $w_{\mathcal{F}_1} \geq w_{\mathcal{F}_2}$, the associated geodesic rays $H_t^{\mathcal{F}_1}$, $H_t^{\mathcal{F}_2}$, $t \in [0, +\infty[$, emanating from $H_0$ compare as $H_t^{\mathcal{F}_1} \leq H_t^{\mathcal{F}_2}$.
		In particular, we have $A(\mathcal{F}_1, H_0) \geq A(\mathcal{F}_2, H_0)$.
	\end{prop}
	
	\paragraph{Mabuchi geodesics.}
	We denote by the $\mathcal{H}_L$ space of positive Hermitian metrics on $L$.
	Upon fixing $h^L_0 \in \mathcal{H}_L$, one can identify $\mathcal{H}_L$ with the space $\mathcal{H}_{\omega}$ of Kähler potentials of $\omega := 2 \pi c_1(L, h^L_0)$, consisting of $u \in \ccal^{\infty}(X, \real)$, such that $\omega_u := \omega + \imun \partial \dbar u$ is strictly positive, through the map 	
	\begin{equation}\label{eq_pot_metr_corr}
		u \mapsto h^L := e^{- u} \cdot h^L_0.
	\end{equation}
	Mabuchi in \cite{Mabuchi} introduced a certain metric on $\mathcal{H}_{\omega}$, the geodesics of which admit the description as solutions to a certain homogeneous Monge-Ampère equation.
	To recall this, we identify paths $u_t \in \mathcal{H}_{\omega}$, $t \in [0, 1]$, with rotationally-invariant $\hat{u} : X \times \mathbb{D}(e^{-1}, 1) \to \real$, as follows
	\begin{equation}\label{eq_defn_hat_u}
		\hat{u}(x, \tau) = u_{t}(x), \quad \text{where} \quad x \in X \, \text{ and } \, t = - \log |\tau|.
	\end{equation}
	According to \cite{Semmes}, \cite{DonaldSymSp} smooth geodesic segments in Mabuchi space can be described as the only path $u_t \in \mathcal{H}_{\omega}$, $t \in [0, 1]$, connecting $u_0$ to $u_1$, so that $\hat{u}$ is the solution of the Dirichlet problem associated with the homogeneous Monge-Ampère equation
	\begin{equation}\label{eq_ma_geod}
		(\pi^* \omega + \imun \partial \dbar \hat{u})^{n + 1} = 0,
	\end{equation}
	with boundary conditions $\hat{u}(x, e^{\imun \theta}) = u_0(x)$, $\hat{u}(x, e^{-1 + \imun \theta}) = u_1(x)$, $x \in X, \theta \in [0, 2\pi]$.
	By the work of X. Chen \cite{ChenGeodMab} and later compliments by B{\l}ocki \cite{BlockiGeod} and Chu-Tosatti-Weinkove \cite{ChuTossVeinC11}, we now know that $\mathscr{C}^{1, 1}$ solutions to (\ref{eq_ma_geod}) always exist.
	\par 
	Berndtsson in \cite[\S 2.2]{BernBrunnMink} also proved that even for $u_0, u_1 \in {\rm{PSH}}(X, \omega) \cap L^{\infty}(X)$,  \textit{weak solutions} to (\ref{eq_ma_geod}) exist, i.e. (\ref{eq_ma_geod}) has solutions when the wedge power is interpreted in Bedford-Taylor sense \cite{BedfordTaylor} and the boundary conditions mean that $\| u_{\epsilon} - u_0 \|_{L^{\infty}(X)} \to 0$ and $\| u_{1 - \epsilon} - u_1 \|_{L^{\infty}(X)} \to 0$, as $\epsilon \to 0$.
	We then have $u_t \in {\rm{PSH}}(X, \omega) \cap L^{\infty}(X)$, but $u_t$ are highly non-regular in general.
	However, since $u_t(x)$ is convex in $t \in [0, 1]$, cf. \cite[Theorem I.5.13]{DemCompl}, the one-sided derivatives $\dot{u}_t^{-}$, $\dot{u}_t^{+}$ of $u_t$ are well-defined for $t \in ]0, 1[$ and they increase in $t$.
	In particular, one can define the derivative of the geodesic segment in this case as $\dot{u}_0 := \lim_{t \to 0} \dot{u}_t^{-} = \lim_{t \to 0} \dot{u}_t^{+}$.
	From \cite[\S 2.2]{BernBrunnMink}, we know that $\dot{u}_0$ is bounded and by Darvas \cite[Theorem 1]{DarvWeakGeod}, we, moreover, have
	\begin{equation}\label{eq_bnd_darvas_sup}
		\sup |\dot{u}_0| \leq \sup |u_1 - u_0|.
	\end{equation}
	Analogously to Theorem \ref{thm_interpol}, we have the following comparison result, which follows directly from the envelope description of the Mabuchi geodesics, cf. \cite[(2.1)]{BernBrunnMink}.
	\begin{prop}\label{prop_comp_geod}
		For any $u_0 \in \mathcal{H}_{\omega}$,  $u_i \in {\rm{PSH}}(X, \omega) \cap L^{\infty}(X)$, $i = 1, 2$, verifying $u_1 \leq u_2$, the geodesics $u_t^1$, $u_t^2$, $t \in [0, 1]$, between $u_0$, $u_1$ and $u_0$, $u_2$ respectively, compare as $u_t^1 \leq u_t^2$.
		In particular, the respective derivatives at zero compare as $\dot{u}^1_0 \leq \dot{u}^2_0$.
	\end{prop}
	
	\paragraph{Test configurations and submultiplicative filtrations.}
	Recall, cf. \cite{TianTestConfig}, \cite{DonaldTestConfig}, that a test configuration $\mathcal{T} = (\pi: \mathcal{X} \to \comp, \mathcal{L})$ for $(X, L)$ consists of the following data
	\begin{enumerate}
		\item A scheme $\mathcal{X}$ with a $\comp^*$-action $\rho$,
		\item A $\comp^*$-equivariant \textit{ample} line bundle $\mathcal{L}$ over $\mathcal{X}$,
		\item A flat $\comp^*$-equivariant projection $\pi : \mathcal{X} \to \comp$, where $\comp^*$ acts on $\comp$ by multiplication, such that if we denote its fibers by $X_{\tau} := \pi^{-1}(\tau)$, $\tau \in \comp$, then $(X_1, \mathcal{L}|_{X_1})$ is isomorphic to $(X, L)$.
	\end{enumerate}
	Remark that the $\comp^*$-action induces the canonical isomorphisms
	\begin{equation}\label{eq_can_ident_test}
		\mathcal{X} \setminus X_0 \simeq \comp^* \times X, \qquad \mathcal{L}|_{\mathcal{X} \setminus X_0} \simeq p^* L,
	\end{equation}
	where $p : \comp^* \times X \to X$ is the natural projection.
	
	\par 
	Following Witt Nystr{\"o}m \cite[Lemma 6.1]{NystOkounTest}, one can construct a submultiplicative filtration  $\mathcal{F}^{\mathcal{T}}$ on $R(X, L)$ associated with $\mathcal{T}$ as follows.
	Pick an element $s \in H^0(X, L^k)$, $k \in \nat^*$, and consider the section $\tilde{s} \in H^0(\mathcal{X} \setminus X_0, \mathcal{L}^k)$, obtained by the application of the $\comp^*$-action to $s$.
	By the flatness of $\pi$, the section $\tilde{s}$ extends to a meromorphic section over $\mathcal{X}$, cf. Witt Nystr{\"o}m \cite[Lemma 6.1]{NystOkounTest}.
	In other words, there is $l \in \integ$, such that for a coordinate $\tau$ on $\comp$, we have $\tilde{s} \cdot \tau^l \in H^0(\mathcal{X}, \mathcal{L}^k)$.
	We define the restriction $\mathcal{F}^{\mathcal{T}}_k$ of the filtration $\mathcal{F}^{\mathcal{T}}$ to $H^0(X, L^k)$ as
	\begin{equation}\label{eq_defn_filt_test}
		\mathcal{F}^{\mathcal{T} \lambda}_k H^0(X, L^k)
		:=
		2 \cdot \Big\{
			s \in H^0(X, L^k) : \tau^{- \lceil \lambda \rceil} \cdot \tilde{s} \in H^0(\mathcal{X}, \mathcal{L}^k)
		\Big\}, 
		\quad \lambda \in \real.
	\end{equation}
	Remark the non-standard normalization by the factor $2$ in (\ref{eq_defn_filt_test}), the motivation for which comes from \cite[Theorem 1.1]{FinTits}.
	One can verify that the filtration $\mathcal{F}^{\mathcal{T}}$ is finitely generated and, up to a restriction to $R(X, L^d) \subset R(X, L)$, for some $d \in \nat$, an arbitrary finitely generated filtration is produced from a test configuration for $(X, L^{\otimes d})$, cf. \cite[(9)]{NystOkounTest} or \cite[\S A.2]{BouckJohn21}.
	
	\paragraph{Construction of geodesic rays in Mabuchi space.}
	Consider the restriction $\pi': \mathcal{X}'_{\mathbb{D}} \to \mathbb{D}$ of a resolution of singularities $\mathcal{T}' := (\pi': \mathcal{X}' \to \comp, \mathcal{L}')$ of a test configuration $\mathcal{T}$ for $(X, L)$ to the unit disc $\mathbb{D}$ and denote $\mathcal{L}'_{\mathbb{D}} := \mathcal{L}'|_{\mathcal{X}'_{\mathbb{D}}}$.
	Phong-Sturm in \cite[Theorem 3]{PhongSturmDirMA} established that for any fixed smooth positive metric $h^L_0$ on $L$, there is a rotation-invariant bounded psh metric $h^{\mathcal{L}'}_{\mathbb{D}}$ over $\mathcal{L}'_{\mathbb{D}}$, solving weakly the Monge-Ampère equation
	\begin{equation}\label{eq_ma_geod_dir}
		c_1(\mathcal{L}'_{\mathbb{D}}, h^{\mathcal{L}'}_{\mathbb{D}})^{n + 1} = 0,
	\end{equation}
	and such that its restriction over $\partial \mathcal{X}'_{\mathbb{D}}$ coincides with the rotation-invariant metric obtained from the fixed metric $h^L_0$ on $L$.
	Under the identification (\ref{eq_can_ident_test}), we then construct a ray $h^{\mathcal{T}}_t$, $t \in [0, + \infty[$, of metrics on $L$, such that $\hat{h}^{\mathcal{T}} = h^{\mathcal{L}'}_{\mathbb{D}}$ in the notations (\ref{eq_defn_hat_u}).
	Due to the equation (\ref{eq_ma_geod_dir}) and the description of the geodesic ray as in (\ref{eq_ma_geod}), we see that the ray of metrics $h^{\mathcal{T}}_t$, $t \in [0, + \infty[$, is a geodesic ray emanating from $h^L_0$.
	This ray of metrics is $\mathscr{C}^{1, 1}$, see \cite{PhongSturmRegul}, \cite{ChuTossVeinC11}.	
	\par 
	More generally, one can construct a geodesic ray from an arbitrary bounded submultiplicative filtration.
	For this, we first recall the well-known construction which associates for an arbitrary Hermitian norm $H_k$ on $H^0(X, L^{\otimes k})$, for $k \in \nat$ so that $L^{\otimes k}$ is very ample, a positive metric $FS(H_k)$  on $L^{\otimes k}$, is constructed for any $l \in L_x^{\otimes k}$, $x \in X$, as follows
	\begin{equation}\label{eq_defn_fs}
		|l|_{FS(H_k)}^{2}
		:=
		\frac{|l|_{h^{L^{\otimes k}}_0}^{2}}{\sum_{i = 1}^{N_k} |s_{i, k}(x)|_{h^{L^{\otimes k}}_0}^{2}},
	\end{equation}
	where $s_{i, k} \in H^0(X, L^{\otimes k})$, is an orthonormal basis of $(H^0(X, L^{\otimes k}), H_k)$.
	\par 
	Now, let $\mathcal{F}$ be a bounded submultiplicative filtration on $R(X, L)$.
	We fix a smooth positive metric $h^L_0$ on $L$, and for any $t \in [0, +\infty[$, $k \in \nat$, we define, following Ross-Witt Nystr{\"o}m \cite{RossNystAnalTConf}, $H^{\mathcal{F}}_{t, k}$ as the (geodesic) ray of Hermitian norms on $H^0(X, L^k)$ emanating from ${\rm{Hilb}}_k(h^L_0)$ and associated with the restriction $\mathcal{F}_k$ of $\mathcal{F}$ to $H^0(X, L^{\otimes k})$.
	We denote by $h^{\mathcal{F}}_t$, $t \in [0, +\infty[$, the ray of metrics on $L$, constructed as follows
	\begin{equation}\label{eq_geod_ray_filt}
		h^{\mathcal{F}}_t := \Big( \lim_{k \to \infty} \inf_{l \geq k} \big( FS(H^{\mathcal{F}}_{t, l})^{\frac{1}{l}} \big) \Big)_{*}.
	\end{equation}
	It was established in \cite{RossNystAnalTConf}, following previous work \cite{PhongSturmTestGeodK} treating the finitely generated case, that $h^{\mathcal{F}}_t$ is a geodesic ray, i.e. it solves (\ref{eq_ma_geod}) weakly, cf. also \cite[Theorem 5.1]{FinNarSim} for an independent proof based on quantization \cite{ChenSunQuant} and Mabuchi geometry \cite{DarLuGeod}.
	In general, however, the metrics $h^{\mathcal{F}}_t$ are only bounded.
	It was established in \cite{PhongSturmRegul} that the two constructions of geodesic rays, (\ref{eq_ma_geod_dir}) and (\ref{eq_geod_ray_filt}), are compatible when the filtration is associated with a test configuration.
	\par 
	It is a well-known open problem to study the precise convergence rate of $FS(H^{\mathcal{F}}_{t, k})^{\frac{1}{k}}$ towards $h^{\mathcal{F}}_t$, see \cite{SongZeldToric} for example.
	Let us put Theorem \ref{thm_berg_conv} in this context.
	Let $x \mapsto B_k^{\mathcal{F}}(x)$ denote the weighted Bergman kernel, as defined in (\ref{eq_weight_berg}) for the function $g(x) := x$. 
	Directly from (\ref{eq_defn_fs}), we see that
	\begin{equation}\label{eq_der_zeld_song}
		FS(H^{\mathcal{F}}_{0, k})^{-\frac{1}{k}} \frac{d}{dt} FS(H^{\mathcal{F}}_{t, k})^{\frac{1}{k}}|_{t = 0}
		=
		- \frac{1}{B_k(x)} B_k^{\mathcal{F}}(x).
	\end{equation}
	Recall now that a well-known result of Tian, \cite{TianBerg}, says
	\begin{equation}\label{eq_thm_tian}
		\frac{1}{k^n} B_k(x) \quad \text{converges uniformly to 1, as } k \to \infty,
	\end{equation}
	see also \cite{ZeldBerg}, \cite{Caltin}, \cite{Bouche}, \cite{MaHol} for more refined results.
	In particular, Theorem \ref{thm_berg_conv} can be interpreted as a statement about convergence of the derivative of the initial point of the ray $FS(H^{\mathcal{F}}_{t, k})^{\frac{1}{k}}$ towards $\dot{h}^{\mathcal{F}}_0$, as $k \to \infty$, giving a partial justification for the $\mathscr{C}^1$-convergence.
	
	\paragraph{Example 1.} 
	The goal of the following example is to show that the assumption of finite degeneration in Theorem \ref{thm_berg_conv} is necessary and cannot be replaced by the regularity assumption on the geodesic ray associated with the filtration.
 	The results of Section \ref{sect_func_calcul} will also imply that the weight operator associated with the filtration constructed here is not a Toeplitz operator, and so the assumptions of generation in Theorem \ref{thm_main1} is necessary as well.
 	\par 
 	We consider the projective space $(X, L) := (\mathbb{P}^1, \mathscr{O}(1))$, and the filtration $\mathcal{F}$ associated the weight function $w_{\mathcal{F}_k}(s) := k \min\{ {\rm{ord}}_0 (s), 1 \}$, where ${\rm{ord}}_0 (s)$ is the order of vanishing of $s \in H^0(\mathbb{P}^1, \mathscr{O}(k))$ at the point $0 := [1, 0] \in \mathbb{P}^1$.
	A straightforward verification reveals that the filtration $\mathcal{F}$ is submultiplicative and bounded, but not finitely generated.
	\par 
	We identify $\mathbb{P}^1$ to $\mathbb{P}(V^*)$, where $V$ is a vector space generated by two elements: $x$ and $y$.
	Let us consider a metric $H$ on $V$, which makes $x$ and $y$ an orthonormal basis, and denote by $h^{FS}$ the induced Fubini-Study metric on $\mathscr{O}(1)$.
	For any $k \in \nat^*$, $i, j \in \nat$, $i + j = k$, under the isomorphism ${\rm{Sym}}^k(V) \to H^0(\mathbb{P}(V^*), \mathscr{O}(k))$, an easy calculation shows that we have
	\begin{equation}\label{eq_l2_norm_calc}
		\big\| x^i \cdot y^j \big\|_{{\rm{Hilb}}_k(h^{FS})}^2
		=
		\frac{i! j!}{(k + 1)!}.
	\end{equation}
	For the weight operator of $\mathcal{F}$, we then obtain the following formula 
	\begin{equation}
		A(\mathcal{F}_k, {\rm{Hilb}}_k(h^{FS})) (x^i \cdot y^j)
		=
		k (1 - \delta_{i, 0})
		x^i \cdot y^j,
	\end{equation}
	where the Kronecker symbol $\delta_{i, 0}$ is defined so that $\delta_{i, 0} = 1$ if $i = 1$, and $\delta_{i, 0} = 0$ otherwise. 
	\par 
	We then denote by $H^{\mathcal{F}}_{t, k}$ the geodesic ray departing from ${\rm{Hilb}}_k(h^{FS})$ and associated with $\mathcal{F}_k$.
	For any $a, b \in \comp$, not simultaneously equal to zero, we have
 	\begin{equation}\label{eq_expl_exmpl11}
 		\frac{FS({\rm{Hilb}}_{k}(h^{FS}))}{FS(H^{\mathcal{F}}_{t, k})} \Big([a x^* + b y^*] \Big)
 		=
 		\frac{e^{tk} (|a|^2 + |b|^2)^{k} + (1 - e^{tk}) |a|^{2k}}{(|a|^2 + |b|^2)^{k}}.
 	\end{equation}
 	In particular, for any $t \in [0, +\infty[$, we conclude that 
 	\begin{equation}
 		\lim_{k \to \infty} \bigg( \frac{FS({\rm{Hilb}}_{k}(h^{FS}))}{FS(H^{\mathcal{F}}_{t, k})} \bigg)^{\frac{1}{k}} \Big([a x^* + b y^*] \Big)
 		=
 		\begin{cases}
 			e^t, \quad &\text{if} \quad b \neq 0,
 			\\
 			1, \quad &\text{otherwise}.
 		\end{cases}
 	\end{equation}
 	Lower semicontinuous regularization will give $h^{\mathcal{F}}_t = e^{-t} h^{FS}$, which implies that $\phi(h^{FS}, \mathcal{F}) = 1$.
 	\par 
 	Let us calculate the weighted Bergman kernel for the function $g(x) = x$.
 	We see directly that 
 	\begin{equation}
 		B_k^{\mathcal{F}, g}([a x^* + b y^*]) 
 		=
 		\sum_{i = 1}^{k} \frac{(k + 1)!}{i! (k - i)!} \frac{|a|^{2i} |b|^{2k - 2i}}{(|a|^2 + |b|^2)^k} 
 		= 
 		(k + 1)
 		\frac{(|a|^2 + |b|^2)^k - |a|^{2k}}{(|a|^2 + |b|^2)^k}.
 	\end{equation}
 	In particular, we deduce that 
 	\begin{equation}
		\lim_{k \to \infty} \frac{1}{k}  B_k^{\mathcal{F}, g}\big([a x^* + b y^*] \big)
 		=
 		\begin{cases}
 			1, \quad &\text{if} \quad b \neq 0,
 			\\
 			0, \quad &\text{otherwise},
 		\end{cases}
 	\end{equation}
 	showing in particular that the convergence towards $\phi(h^{FS}, \mathcal{F})$ is neither uniform nor even pointwise, despite the regularity and simplicity of the associated geodesic ray, $h^{\mathcal{F}}_t$.

	\paragraph{Example 2.}
	The goal of the following example is to show that the uniform convergence in Theorem \ref{thm_berg_conv} cannot be improved to the $\mathscr{C}^1$-convergence.
	\par 
	We consider the projective space $(X, L) := (\mathbb{P}^1, \mathscr{O}(2))$, and the filtration $\mathcal{F}$ associated the weight function $w_{\mathcal{F}_k}(s) := \min\{ {\rm{ord}}_0 (s), k \}$ in the notations of the previous example.
	The reader will check that $\mathcal{F}$ is finitely generated.
	Similar calculation to the ones behind (\ref{eq_expl_exmpl11}) will reveal that for any $a, b \in \comp$, not simultaneously equal to zero, we have
 	\begin{multline}\label{eq_expl_exmpl112}
 		\frac{FS({\rm{Hilb}}_{2k}(h^{FS}))}{FS(H^{\mathcal{F}}_{t, k})} \Big([a x^* + b y^*] \Big)
 		=
 		\frac{1}{(|a|^2 + |b|^2)^{2k}}
 		\Big(
 		\sum_{i = 0}^{k} e^{ti} |a|^{2i} |b|^{2(2k - i)} \frac{(2k)!}{i! (2k - i)!}
 		\\
 		+
 		e^{tk}
 		\cdot
 		\sum_{i = k + 1}^{2k} |a|^{2i} |b|^{2(2k - i)} \frac{(2k)!}{i! (2k - i)!}
 		\Big).
 	\end{multline}
 	\par Cramér's theorem from large deviations theory applied for the binomial distribution yields that for any $x < 1$, we have
 	\begin{equation}\label{eq_cramer}
 		\lim_{k \to \infty} \frac{1}{2 k} \log \Big(
 		\sum_{i = k + 1}^{2k} x^i \frac{(2k)!}{i! (2k - i)!}
 		\Big)
 		=
 		\frac{1}{2} \log (4 x).
 	\end{equation}
 	From (\ref{eq_cramer}) and binomial formula, it is immediate to recover that for any $t \in [0, +\infty[$, we have
 	\begin{equation}
 		\frac{(h^{FS})^2}{h^{\mathcal{F}}_t} \Big([a x^* + b y^*] \Big)
 		=
 		\begin{cases}
 			\big(\frac{e^{t} |a|^2 + |b|^2}{|a|^2 + |b|^2} \big)^2,
 			&
 			e^{t / 2}|a| < |b|,
 			\\
 			\big(\frac{2 e^{t / 2} |a| |b|}{|a|^2 + |b|^2} \big)^2,
 			&
 			e^{- t / 2} |b| < |a| < |b|,
 			\\
 			e^{t},
 			&
 			|b| < |a|.
 		\end{cases}
 	\end{equation}
 	By differentiation, we obtain
 	\begin{equation}
 		\phi(h^{FS}, \mathcal{F}) \big([a x^* + b y^*] \big)
 		=
 		\begin{cases}
 			\frac{2 |a|^2}{|a|^2 + |b|^2},
 			&
 			|a| < |b|,
 			\\
 			1,
 			&
 			|b| < |a|.
 		\end{cases}
 	\end{equation}
 	One can easily see that $\phi(h^{FS}, \mathcal{F})$ is Lipshitz but not $\mathscr{C}^1$.
 	As for each $k \in \nat^*$, the function $x \mapsto \frac{1}{k^n} B_k^{\mathcal{F}, g}(x)$, $x \in X$, is clearly $\mathscr{C}^1$, the uniform convergence cannot be improved to the $\mathscr{C}^1$-convergence for the above filtration.
 	Our example here is of course related to the well-known phenomena that one can expect at most $\mathscr{C}^{1, 1}$-regularity for the geodesic rays, cf. \cite{BermanEnvProj}, \cite{ChenTangAster}, \cite{SongZeldToric}.

\section{Asymptotics of weight operators, proofs of Theorems \ref{thm_main1}, \ref{thm_main2}}\label{sect_pf}

The primary objective of this section is to demonstrate Theorems \ref{thm_main1} and \ref{thm_main2}, subject to an auxiliary result that will be established in the subsequent section of the article.
\par 
	We will fix a positive smooth metric $h^L_0$ on $L$, a continuous psh metric $h^L_1$ on $L$, and Kähler forms $\chi_i$, $i = 0, 1$, on $X$.
	We denote by $T_k(h^L_0, h^L_1) \in {\rm{End}}(H^0(X, L^{\otimes k}))$ the transfer map between ${\textrm{Hilb}}_k(h^L_0, \chi_0)$ and ${\textrm{Hilb}}_k(h^L_1, \chi_1)$.
	Let $h^L_t$, $t \in [0, 1]$, be the Mabuchi geodesic between $h^L_0$ and $h^L_1$.
	We denote $\phi(h^L_0, h^L_1) := - \dot{h}^L_0$.
	By Demailly's regularization theorem \cite{DemRegul}, cf. \cite{GuedZeriGeomAnal}, for any $\epsilon > 0$, we can find a positive Hermitian metric $h^L_{1, \epsilon}$ on $L$, verifying $h^L_1 \exp(-\epsilon) \leq h^L_{1, \epsilon} \leq h^L_1 \exp(\epsilon)$.
	\begin{prop}\label{prop_cont_der}
		As $\epsilon \to 0$, the sequence of functions $\phi(h^L_0, h^L_{1, \epsilon})$ converges to $\phi(h^L_0, h^L_1)$ uniformly.
		In particular, $\phi(h^L_0, h^L_1)$ is a continuous function.
	\end{prop}
	\begin{proof}
		By Proposition \ref{prop_comp_geod}, we have
		\begin{equation}\label{eq_est_der_2_0}
			\phi(h^L_0, h^L_{1, \epsilon}) - \epsilon
			\leq
			\phi(h^L_0, h^L_1)
			\leq
			\phi(h^L_0, h^L_{1, \epsilon}) + \epsilon.
		\end{equation}
		The result follows directly from (\ref{eq_est_der_2_0}) since $\phi(h^L_0, h^L_{1, \epsilon})$ is continuous for any $\epsilon > 0$ by \cite{PhongSturmRegul}.
	\end{proof}
	\par 
	The following result, established in Section \ref{sect_transfer}, is among the main results of this paper.
	\begin{thm}\label{thm_trasnfer}
		The operators $\{ \frac{1}{k} T_k(h^L_0, h^L_1) \}_{k = 1}^{+ \infty}$, form a Toeplitz operator with symbol $\phi(h^L_0, h^L_1)$.
	\end{thm}
	\par 
	We establish Theorem \ref{thm_main1} by using Theorem \ref{thm_trasnfer}.
	In order to make a connection between the transfer operator and the weight operator, we need to relate the associated rays of norms on $R(X, L)$.  
	Let $\mathcal{F}$ be a finitely generated submultiplicative filtration on $R(X, L)$.
	We define $h^{\mathcal{F}}_t$, $t \in [0, +\infty[$, and $H^{\mathcal{F}}_{t, k}$, $t \in [0, +\infty[$, $k \in \nat$, as in (\ref{eq_geod_ray_filt}).
	We also fix an arbitrary Kähler form $\chi$ on $X$.
	The following result is at the heart of our approach.
	\begin{thm}[{\cite[Theorem 4.1]{FinTits}}]\label{thm_comp_geod}
		There are $C > 0$, $k_0 \in \nat^*$, such that for any $t \in [0, +\infty[$, $k \geq k_0$, we have the following comparison of norms
		\begin{equation}\label{eq_thm_2_step1}
			\exp(-C(t + k))
			\cdot
			{\rm{Hilb}}_k(h^{\mathcal{F}}_t, \chi)
			\leq
			H^{\mathcal{F}}_{t, k}
			\leq
			\exp(C(t + k))
			\cdot
			{\rm{Hilb}}_k(h^{\mathcal{F}}_t, \chi).
		\end{equation}
	\end{thm}
	\begin{rem}
		In \cite[Theorem 4.1]{FinTits}, we only established Theorem \ref{thm_comp_geod} for filtrations $\mathcal{F}$ associated with a test configuration.
		Due to the correspondence between test configurations and finitely generated submultiplicative filtrations, recalled in Section \ref{sect_prel}, it implies that there is $d \in \nat^*$, such that (\ref{eq_thm_2_step1}) holds for any $k$ divisible by $d$.
		However, the same argument as in \cite[proof of Theorem 4.1]{FinTits}, based on the isometry properties of the surjective map $H^0(X, L^{\otimes k}) \otimes H^0(X, L^{\otimes l}) \to H^0(X, L^{\otimes (k + l)})$ then gives (\ref{eq_thm_2_step1}) for all $k \in \nat$, see more specifically \cite[(4.12), (4.14) and (4.18)]{FinTits}.
	\end{rem}
	We can now present the proof of the main result of this paper.
	\begin{proof}[Proof of Theorem \ref{thm_main1}]
		We conserve the notations from Theorems \ref{thm_main1} and \ref{thm_comp_geod}.
		By Theorem \ref{thm_comp_geod} and Proposition \ref{thm_interpol}, for any $k \geq k_0$, $t \in ]0, +\infty[$, we have the following relation between the weight operator and the transfer operators
		\begin{equation}\label{eq_comp_weight_transfer}
			 \Big\|
				A(\mathcal{F}_k, {\textrm{Hilb}}_k(h^L))
				-
				\frac{1}{t} 
				T_k(h^{\mathcal{F}}_0, h^{\mathcal{F}}_t)
			\Big\|
			\leq
			C \frac{k}{t} + C.
		\end{equation}
		Theorem \ref{thm_main1} follows directly from Theorem \ref{thm_trasnfer} and the trivial fact $\phi (h^{\mathcal{F}}_0, h^{\mathcal{F}}_t) = t \phi (h^{\mathcal{F}}_0, h^{\mathcal{F}}_1)$.
	\end{proof}
	\par 
	We will now deduce Theorem \ref{thm_main2} from Theorem \ref{thm_main1} and approximation techniques for submultiplicative filtrations from \cite{SzekeTestConf}, \cite{BerBouckJonYTD}, \cite{BouckJohn21}, \cite{FinSecRing}.
	For this, we will need an auxillary statement, which will be established in Section \ref{sect_func_calcul}.
	\begin{prop}\label{prop_conv_toepl_symb}
		For an arbitrary sequence $f_k \in L^{\infty}(X)$, $k \in \nat$, of uniformly bounded functions converging, as $k \to \infty$, in $L^1(X)$ towards $f \in L^{\infty}(X)$, the sequence of operators $\{ T_{f_k, k} \}_{k = 0}^{+\infty}$ forms a Toeplitz operator of Schatten class with symbol $f$.
	\end{prop}
	\par
	\begin{proof}[Proof of Theorem \ref{thm_main2}]
	First of all, without loosing the generality we can assume that the filtration $\mathcal{F}$ has integer weights.
	In order to see this, for an arbitrary filtration $\mathcal{F}$, we define the filtration $\lfloor \mathcal{F} \rfloor$ through its weight function as $w_{\lfloor \mathcal{F} \rfloor}(s) := \lfloor w_{ \mathcal{F} }(s) \rfloor$, $s \in R(X, L)$. It is easy to see that $\lfloor \mathcal{F} \rfloor$ is submultiplicative and bounded as well.
	Moreover, by Theorem \ref{thm_interpol}, we have 
	\begin{equation}
		A(\mathcal{F}_k, {\textrm{Hilb}}_k(h^L)) 
		\geq
		A(\lfloor \mathcal{F}_k \rfloor, {\textrm{Hilb}}_k(h^L)) 
		\geq
		A(\mathcal{F}_k, {\textrm{Hilb}}_k(h^L)) - {\rm{Id}}.
	\end{equation}
	Which shows that it suffices to establish Theorem \ref{thm_main2} for $\lfloor \mathcal{F} \rfloor$. 
	Without loosing the generality, we assume from now on that the filtration $\mathcal{F}$ has integer weights.
	\par 
	Now, as $R(X, L)$ is finitely generated, there is $k_0 \in \nat$, such that $\oplus_{i = 0}^{k_0} H^0(X, L^{\otimes i})$ generates $R(X, L)$ as a ring. 
	For any $k \geq k_0$, we denote by $\mathcal{F}^{[k]}$ the filtration on $R(X, L)$, generated by the restriction of $\mathcal{F}$ to $\oplus_{i = 0}^{k_0} H^0(X, L^{\otimes i})$.
	Clearly, $\mathcal{F}^{[k]}$ has integer weights and is finitely generated.
	By Theorem \ref{thm_main1}, for any $\epsilon > 0$, $k \geq k_0$, there is $l_0 \in \nat$, such that for any $l \geq l_0$, we have
	\begin{equation}\label{eq_weight_approx_beha}
		\Big\| 
			\frac{1}{l} A( \mathcal{F}^{[k]}_l, {\textrm{Hilb}}_l(h^L))
			-
			T_{\phi(h^L, \mathcal{F}^{[k]}), l} 
		\Big\|
		\leq
		\epsilon.
	\end{equation}
	\par 
	Our proof resides on establishing the following two statements. 
	First, we establish that 
	\begin{equation}\label{eq_conv_deriv1}
		\phi(h^L, \mathcal{F}^{[k]}), k \in \nat \text{ increases to } \phi(h^L, \mathcal{F}) \text{ almost everywhere}.
	\end{equation}
	Second, we establish that for any $\epsilon > 0$, $p \in [1, +\infty[$, there is $k_1 \geq k_0$, such that for any $k \geq k_1$, there is $l_1 \in \nat$, so that for any $l \geq l_0$, we have
	\begin{equation}\label{eq_weight_oper_appr}
		\Big\| 
			A( \mathcal{F}^{[k]}_l, {\textrm{Hilb}}_l(h^L))
			-
			A( \mathcal{F}_l, {\textrm{Hilb}}_l(h^L))
		\Big\|_p
		\leq
		\epsilon l.
	\end{equation}
	Remark that Theorem \ref{thm_main2} follows directly from Proposition \ref{prop_conv_toepl_symb}, (\ref{eq_weight_approx_beha}), (\ref{eq_conv_deriv1}) and (\ref{eq_weight_oper_appr}) by Lebesgue dominated convergence theorem.
	From now on, we concentrate on the proofs of (\ref{eq_conv_deriv1}) and (\ref{eq_weight_oper_appr}).
	\par 
	\begin{sloppypar}
	Remark first that as an immediate consequence of submultiplicativity of $\mathcal{F}$, the weight functions of the filtrations are related for any $k \geq k_0$, $l \in \nat$, as follows
	\begin{equation}\label{eq_weight_monotone}
		w_{\mathcal{F}^{[k]}_l} \leq w_{\mathcal{F}^{[k + 1]}_l} \leq w_{\mathcal{F}_l}.
	\end{equation}
	By Theorem \ref{thm_interpol} and (\ref{eq_weight_monotone}), we deduce that for any $k \geq k_0$, $l \in \nat^*$, we have
	\begin{equation}\label{eq_monotone_ak}
		A( \mathcal{F}^{[k]}_l, {\textrm{Hilb}}_l(h^L))
		\leq
		A( \mathcal{F}^{[k + 1]}_l, {\textrm{Hilb}}_l(h^L))
		\leq
		A( \mathcal{F}_l, {\textrm{Hilb}}_l(h^L)),
	\end{equation}
	By Theorem \ref{thm_interpol}, (\ref{eq_geod_ray_filt}) and (\ref{eq_weight_monotone}), we deduce that in the notations of (\ref{eq_geod_ray_filt}), for any $k \in \nat$, $t \in [0, +\infty[$, we have
	\begin{equation}\label{eq_incr_phi00}
		h^{\mathcal{F}^{[k]}}_t \geq h^{\mathcal{F}^{[k + 1]}}_t \geq h^{\mathcal{F}}_t.
	\end{equation}
	By taking derivatives at $t = 0$ in (\ref{eq_incr_phi00}), for any $k \in \nat$, we establish 
	\begin{equation}\label{eq_incr_phi}
		\phi(h^L, \mathcal{F}^{[k]}) \leq \phi(h^L, \mathcal{F}^{[k + 1]}) \leq \phi(h^L, \mathcal{F})
	\end{equation}
	\par 
	Let $k_2 \in \nat$ is defined so that for any $l_1, l_2 \geq k_2$, the multiplication maps $H^0(X, L^{\otimes l_1}) \otimes H^0(X, L^{\otimes l_2}) \to H^0(X, L^{\otimes (l_1 + l_2)})$ is surjective.
	Recall, cf. \cite[Proposition 3.2.6]{ChenHNolyg}, that for an arbitrary submultiplicative filtration $\mathcal{F}$, $l_1, l_2 \geq k_2$, we have 
	\begin{equation}
		\min_{s \in H^0(X, L^{\otimes l_1})} w_{\mathcal{F}_{l_1}}(s)
		+
		\min_{s \in H^0(X, L^{\otimes l_2})} w_{\mathcal{F}_{l_2}}(s)
		\leq
		\min_{s \in H^0(X, L^{\otimes (l_1 + l_2)})} w_{\mathcal{F}_{l_1 + l_2}}(s).
	\end{equation}
	From this and (\ref{eq_weight_monotone}), we see in particular that the following bound holds
	\begin{equation}\label{eq_bound_max_vol_appr}
		\sup_{k \geq \max( k_0, k_2)}  \sup_{l \in \nat^*} \frac{1}{l} \max_{s \in H^0(X, L^{\otimes l}) \setminus \{0\}} |w_{\mathcal{F}^{[k]}_l}(s)|
		\leq
		\sup_{l \in \nat^*} \frac{1}{l} \max_{s \in H^0(X, L^{\otimes l}) \setminus \{0\}} |w_{\mathcal{F}_l}(s)|.
	\end{equation}
	\end{sloppypar}
	\par 
	\begin{sloppypar}
	From \cite[Lemma 2.4]{FinTits}, using the boundness of $\mathcal{F}$, for $C := \sup_{l \in \nat^*} \frac{1}{l} \max |w_{\mathcal{F}_l}|$, we have $\phi(h^L, \mathcal{F}) \leq C$.
	From this and (\ref{eq_incr_phi}), we see that $\phi(h^L, \mathcal{F}^{[k]})$ is uniformly bounded in $k \in \nat$.
	Due to this and (\ref{eq_incr_phi}), in order to get (\ref{eq_conv_deriv1}), it suffices to establish that $\phi(h^L, \mathcal{F}^{[k]})$ converges to $\phi(h^L, \mathcal{F})$ in $L^1(X)$, as $k \to \infty$.
	\end{sloppypar}
	\par 
	Recall that Darvas-Lu in \cite[Theorem 3.1]{DarLuGeod} established the following bound
	\begin{equation}
		\frac{1}{n! \int c_1(L)^n} \int \big| \phi(h^L, \mathcal{F}^{[k]}) - \phi(h^L, \mathcal{F}) \big| c_1(L, h^L)^n 
		\leq
		d_1(h^{\mathcal{F}^{[k]}}_1, h^{\mathcal{F}}_1),
	\end{equation}
	where $d_1$ is the Darvas distance, \cite{DarvasMabCompl}.
	In particular, to show the convergence of $\phi(h^L, \mathcal{F}^{[k]})$ to $\phi(h^L, \mathcal{F})$ in $L^1(X)$, as $k \to \infty$, it only suffices to show that $d_1(h^{\mathcal{F}^{[k]}}_1, h^{\mathcal{F}}_1) \to 0$, as $k \to \infty$.
	This was done in \cite[Theorem 5.9]{FinSecRing} for a similar approximation scheme, where instead of the filtration $\mathcal{F}^{[k]}$, we considered the filtration $\mathcal{F}^{(k)}$ on $R(X, L^{\otimes k})$ induced by $\mathcal{F}_k$. 
	Let us explain why this implies the searched convergence. 
	Indeed, as an immediate consequence of submultiplicativity of $\mathcal{F}$, similarly to (\ref{eq_weight_monotone}), the weight functions of the filtrations are related for any $k \geq k_0$, $l \in \nat$, as
	\begin{equation}\label{eq_weight_monotone121}
		w_{\mathcal{F}^{(k)}_l} \leq w_{\mathcal{F}^{[k]}_{kl}} \leq w_{\mathcal{F}_{kl}}.
	\end{equation}
	We denote by $h^{\mathcal{F}^{(k)}}_t$, $t \in [0, +\infty[$, the geodesic ray on $L^{\otimes k}$ emanating from $(h^L)^{\otimes k}$ and induced by $\mathcal{F}^{(k)}$.
	From (\ref{eq_weight_monotone121}), we deduce in the same way as in (\ref{eq_incr_phi00}) that
	\begin{equation}\label{eq_incr_phi00121}
		(h^{\mathcal{F}^{(k)}}_t)^{\frac{1}{k}} \geq h^{\mathcal{F}^{[k]}}_t \geq h^{\mathcal{F}}_t.
	\end{equation}
	The statement \cite[Theorem 5.9]{FinSecRing} says that $d_1((h^{\mathcal{F}^{(k)}}_1)^{\frac{1}{k}}, h^{\mathcal{F}}_1) \to 0$, as $k \to \infty$.
	However, by (\ref{eq_incr_phi00121}) and the usual properties of the Darvas distance, cf. \cite[Lemma 4.2]{DarvasFinEnerg}, we have $d_1((h^{\mathcal{F}^{(k)}}_1)^{\frac{1}{k}}, h^{\mathcal{F}}_1) \geq d_1(h^{\mathcal{F}^{[k]}}_1, h^{\mathcal{F}}_1)$, which establishes (\ref{eq_conv_deriv1}).
	\par 
	We will now establish (\ref{eq_weight_oper_appr}).
	Remark that due to (\ref{eq_bound_max_vol_appr}), there is $C > 0$, so that for any $k \geq \max( k_0, k_2)$, $l \in \nat^*$, we have
	\begin{equation}\label{eq_norm_a_bnd_unif}
		\Big\| 
			A( \mathcal{F}^{[k]}_l, {\textrm{Hilb}}_l(h^L))
			-
			A( \mathcal{F}_l, {\textrm{Hilb}}_l(h^L))
		\Big\|
		\leq
		C l.
	\end{equation}
	Remark also the classical fact that for an arbitrary operator $A_k \in {\rm{End}}(H^0(X, L^{\otimes k}))$, we have 
	\begin{equation}\label{eq_a_k_bnd_p1}
		\| A_k \|_p \leq \| A_k \|_1^{\frac{1}{p}} \cdot \| A_k \|^{1 - \frac{1}{p}}.
	\end{equation}
	From (\ref{eq_norm_a_bnd_unif}) and (\ref{eq_a_k_bnd_p1}), we see that it suffices to establish (\ref{eq_weight_oper_appr}) for $p = 1$.
	\par 
	Directly from (\ref{eq_monotone_ak}), we have
	\begin{multline}\label{eq_norm_trace_appr}
		\Big\| 
			A( \mathcal{F}^{[k]}_l, {\textrm{Hilb}}_l(h^L))
			-
			A( \mathcal{F}_l, {\textrm{Hilb}}_l(h^L))
		\Big\|_1
		\\
		=
		\frac{1}{N_k} \Big(
		\tr{A( \mathcal{F}_l, {\textrm{Hilb}}_l(h^L))}
		-
		\tr{A( \mathcal{F}^{[k]}_l, {\textrm{Hilb}}_l(h^L))} \Big).
	\end{multline}
	\par 
	\begin{sloppypar}
	Recall that the volume, ${\rm{vol}}(\mathcal{F})$, of a submultiplicative filtration $\mathcal{F}$ is defined as follows
	\begin{equation}
		{\rm{vol}}(\mathcal{F})
		:=
		\lim_{k \to \infty} \frac{n!}{k^{n + 1}} \sum_{i = 1}^{N_k} e_{\mathcal{F}}(i, k), 
	\end{equation}
	where $e_{\mathcal{F}}(i, k)$ are the \textit{jumping numbers} of the filtration $\mathcal{F}_k$ on $H^0(X, L^{\otimes k})$, defined in (\ref{eq_defn_jump_numb}).
	\end{sloppypar}
	\par 
	By (\ref{eq_basic_id}), (\ref{eq_norm_trace_appr}) and the above discussion, we see that (\ref{eq_weight_oper_appr}) would follow if we establish that ${\rm{vol}}(\mathcal{F}^{[k]}) \to {\rm{vol}}(\mathcal{F})$.
	For a similar approximation scheme, where instead of the filtration $\mathcal{F}^{[k]}$, we consider the filtration $\mathcal{F}^{(k)}$ defined above, this was done by Boucksom-Jonsson in \cite[Theorem 3.18 and (3.14)]{BouckJohn21} and later by the author in \cite[Theorem 5.10]{FinSecRing}. 
	However, by proceeding as in (\ref{eq_incr_phi00121}), we easily see that this implies that the statement also holds for $\mathcal{F}^{[k]}$.
	\end{proof}

\section{Transfer operator between $L^2$-norms, a proof of Theorem \ref{thm_trasnfer}}\label{sect_transfer}
	In this section, we describe our proof of  Theorem \ref{thm_trasnfer}.
	Our approach is inspired by Berndtsson's previous work \cite{BerndtProb}, in the sense that it also relies on the connection between geodesics in the space of Hermitian metrics and the curvature of the associated Hermitian vector bundles. 
	However, we diverge from it in several important ways.
	\par 
	First, by using the curvature calculations of $L^2$-metrics due to Ma-Zhang \cite{MaZhangSuperconnBKPubl}, we demonstrate that $L^2$-metrics can be used not only to construct superinterpolating families for the geodesic between two $L^2$-metrics, as in \cite{BerndtProb}, but also subinterpolating ones. 
	This enables a much more refined analysis of the transfer operator, which gives an even stronger version of Theorem \ref{thm_trasnfer} when the Mabuchi geodesic is smooth and passes through positive metrics.
	\par 
	Second, \cite{MaZhangSuperconnBKPubl} only works for smooth families of metrics, so we need to regularize our geodesic, which is not smooth in general. 
	While a similar regularity issue arises in \cite{BerndtProb} and is addressed there by selecting an \textit{arbitrary} smooth subgeodesic, the specific choice of subgeodesic is critical in our approach. 
	In particular, the construction of $\epsilon$-geodesics from \cite{ChenGeodMab} is crucial to our method.
	However, due to the lack of non-collapsing estimates for $\epsilon$-geodesics, see the discussion after (\ref{eq_r_k_eps}), this inevitably introduces additional challenges. 
	Specifically, for non-smooth Mabuchi geodesics, we can only show that $L^2$-metrics form subinterpolating families over a small interval that contains an endpoint.
	Relying on our previous work \cite{FinSecRing}, we nevertheless show that this is enough to establish Theorem \ref{thm_trasnfer}.
	\par 
	We will now proceed as follows.
	We first establish Theorem \ref{thm_trasnfer} for smooth endpoints, verifying an additional assumption that the Mabuchi geodesic between them is smooth and passes through positive metrics.
	Remark that it was established by Lempert-Vivas \cite{LempViv} and Darvas-Lempert \cite{DarvLempWeakReg} that for general endpoints, there is no smooth Mabuchi geodesic relating them.
	We then establish Theorem \ref{thm_trasnfer} for arbitrary smooth endpoints.
	By approximation, we establish Theorem \ref{thm_trasnfer} in its full generality.
	\par 
	Let us recall first a classical result making a connection between the geodesic construction in the space $\mathcal{H}_V$ of Hermitian norms on a complex vector space $V$ and positivity properties of related vector bundles.
	For this, similarly to (\ref{eq_defn_hat_u}), we identify a smooth path $H_t \in \mathcal{H}_V$, $t \in [0, 1]$, with the rotationally-invariant Hermitian metric $\hat{H}$ on the (trivial) vector bundle $V \times \mathbb{D}(e^{-1}, 1)$ over $\mathbb{D}(e^{-1}, 1)$, through the formula 
	\begin{equation}\label{eq_defn_hat_H}
		\hat{H}(\tau) = H_{t}, \quad \text{where} \quad t = - \log |\tau|.
	\end{equation}
	\par 
	Recall that a Hermitian vector bundle $(E, h^E)$ over a Riemann surface $S$ is called positive, cf. \cite[\S VII.6]{DemCompl}, if the curvature $R^E$ of the Chern connection of $(E, h^E)$, for any $s \in S$, can be written as $R^E_s = dz \wedge d\overline{z} \cdot A(s)$, for a positively definite $A(s) \in {\rm{End}}(E_s)$ and a local holomorphic coordinate $z$ on $S$, centered at $s$.
	\begin{thm}[ {\cite[Theorem 4.2]{RochbergInter}, \cite[Theorem 4.1, \S 15]{CoifmanSemmes} }]\label{thm_curv_geod}
		Assume that a smooth path $H_t^0 \in \mathcal{H}_V$, $t \in [0, 1]$, is such that the associated Hermitian metric $\hat{H}_0$ on $V \times \mathbb{D}(e^{-1}, 1)$ is positive (resp. negative). 
		Then for the geodesic $H_t$ between $H_t^0$ and $H_t^1$, we have $H_t^0 \geq H_t$ (resp. $H_t^0 \leq H_t$).
		In particular, the following inequality is satisfied  $\dot{H}_0^0 \geq \dot{H}_0$ (resp. $\dot{H}_0^0 \leq \dot{H}_0$).
		The path $H_t^0$ is then called a superinterpolating (resp. subinterpolating) family.
	\end{thm}
	\par 
	We will also need the result of Ma-Zhang \cite{MaZhangSuperconn}, \cite{MaZhangSuperconnBKPubl} on the curvature of direct image bundles. 
	We fix a smooth family of (strictly) positive Hermitian metrics $h^L_{\tau}$, $\tau \in \mathbb{D}(e^{-1}, 1)$ on $L$ and a smooth family of Kähler forms $\chi_{\tau}$, $\tau \in \mathbb{D}(e^{-1}, 1)$, on $X$.
	We denote by $\omega := c_1(L \times \mathbb{D}(e^{-1}, 1), h^L_{\tau})$ the curvature of $h^L_{\tau}$, viewed as a metric on the line bundle $L \times \mathbb{D}(e^{-1}, 1)$ over $X \times \mathbb{D}(e^{-1}, 1)$.
	\par 
	For $\tau \in \mathbb{D}(e^{-1}, 1)$, we define $\omega_H(\tau) \in \ccal^{\infty}(X)$ as 
	\begin{equation}
		\omega_H(\tau)(x) := \frac{1}{n + 1} \frac{\omega^{n + 1}}{\omega^{n} \wedge \imun dz \wedge d \overline{z}}(x, \tau).
	\end{equation}
	The denominator above is nonzero, as $\omega$ is positive along the fibers.
	\par 
	Now, we denote by $R_k$ the curvature of the Chern connection on the trivial vector bundle $H^0(X, L^{\otimes k}) \times \mathbb{D}(e^{-1}, 1)$ associated with the fiberwise $L^2$-metric ${\textrm{Hilb}}_k(h^L_{\tau}, \chi_{\tau})$, $\tau \in \mathbb{D}(e^{-1}, 1)$, induced by $h^L_{\tau}$.
	We introduce $D_k(\tau) \in {\rm{End}}(H^0(X, L^{\otimes k}))$ through the identity $\frac{\imun}{2 \pi} R_{k, \tau} := \imun dz \wedge d \overline{z} \cdot D_k(\tau)$.
	The following result is the technical backbone of our analysis.
	\begin{thm}[{ Ma-Zhang \cite[Theorem 0.4]{MaZhangSuperconnBKPubl} }]\label{thm_mazh}
		There are $C > 0$, $k_0 \in \nat$, such that in the notations from Definition \ref{defn_toepl}, for any $k \geq k_0$, $\tau \in \mathbb{D}(e^{-1}, 1)$, we have 
		\begin{equation}
			\Big\|
				D_k(\tau)
				-
				k T_{\omega_H(\tau), k}
			\Big\|
			\leq
			C,
		\end{equation}
		where $\| \cdot \|$ is the operator norm subordinate with ${\textrm{Hilb}}_k(h^L_{\tau}, \chi_{\tau})$.
		In particular, $\{ \frac{1}{k} D_k(\tau) \}_{k = 1}^{+ \infty}$ forms a Toeplitz operator with symbol $\omega_H(\tau)$.
	\end{thm}
	We will now establish a refinement of Theorem \ref{thm_trasnfer} under more restrictive assumptions. 
	We assume that the Mabuchi geodesic $h^L_t$, $t \in [0, 1]$, between two positive smooth metrics $h^L_i$, $i = 0, 1$, is smooth (jointly in $X$ and $t$ directions) and $h^L_t$ are (strictly) positive Hermitian metrics.
	We denote by $T_k(h^L_0, h^L_1) \in {\rm{End}}(H^0(X, L^{\otimes k}))$ the transfer map between ${\textrm{Hilb}}_k(h^L_0, \chi_0)$ and ${\textrm{Hilb}}_k(h^L_1, \chi_1)$, where $\chi_0$, $\chi_1$ are some Kähler forms on $X$.
	\begin{thm}\label{thm_trasnfer_sm}
		In the notations of Definition \ref{defn_toepl} and Theorem \ref{thm_trasnfer}, under the described above assumption, there are $C > 0$, $k_0 \in \nat$, such that for any $k \geq k_0$, we have
		\begin{equation}
			\big\| 
				T_k(h^L_0, h^L_1)
				-
				k T_{\phi(h^L_0, h^L_1), k}
			\big\|
			\leq
			C.
		\end{equation}
	\end{thm}
	\begin{proof}
		We fix a smooth path of Kähler forms $\chi_t$, $t \in [0, 1]$, on $X$ between $\chi_0$ and $\chi_1$, and consider the rotationally-invariant Hermitian metric $\hat{H}^0_k$ on the (trivial) vector bundle $H^0(X, L^{\otimes k}) \times \mathbb{D}(e^{-1}, 1)$ over $\mathbb{D}(e^{-1}, 1)$, constructed from ${\textrm{Hilb}}_k(h^L_t, \chi_t)$, $t \in [0, 1]$, as in (\ref{eq_defn_hat_H}).
		Directly from the fact that smooth Mabuchi geodesics solve the homogeneous Monge-Ampère equation (\ref{eq_ma_geod}), by Theorem \ref{thm_mazh}, we deduce that there are $C_0 > 0$, $k_0 \in \nat$, such that the curvature $R_k$ of $\hat{H}^0_k$ satisfies $\| R_k \| \leq C_0$ for any $k \geq k_0$.
		We denote
		\begin{equation}\label{eq_g_func_defn}
			g: \mathbb{D}(e^{-1}, 1) \to \real, \qquad \tau \mapsto g(\tau) := (2 \log |\tau|^2 - 1)^2 - 1.
		\end{equation}
		Remark that $g$ is strictly subharmonic and verifies $g(e^{-1 + i\theta}) = g(e^{i\theta}) = 0$, for any $\theta \in [0, 2 \pi]$.
		Directly from the bound $\| R_k \| \leq C_0$, and strict subharmonicity of $g$, there is $C_1 > 0$, such that the curvature of Hermitian metrics $\hat{H}^1_k = \hat{H}^0_k \cdot \exp(- C_1 g)$ (resp. $\hat{H}^2_k = \hat{H}^0_k \cdot \exp(C_1 g)$) is positive (resp. negative).
		We denote by $H^1_{k, t}$, $H^2_{k, t}$, $t \in [0, 1]$, the paths of metrics on $H^0(X, L^{\otimes k})$ induced through (\ref{eq_defn_hat_H}) by $\hat{H}^1_k$ and $\hat{H}^2_k$ respectively.
		Our boundary condition on $g$ implies that $H^1_{k, 0} = {\textrm{Hilb}}_k(h^L_0, \chi_0) = H^2_{k, 0}$ and $H^1_{k, 1} = {\textrm{Hilb}}_k(h^L_1, \chi_1) = H^2_{k, 1}$.
		From this and the above curvature calculation, we deduce by Theorem \ref{thm_curv_geod} that for any $t \in [0, 1]$, we have
		\begin{equation}
			H^1_{k, t} \geq H_{k, t} \geq H^2_{k, t},
		\end{equation}
		where $H_{k, t}$ is the geodesic between ${\textrm{Hilb}}_k(h^L_0, \chi_0)$ and ${\textrm{Hilb}}_k(h^L_1, \chi_1)$.
		By taking derivatives at $t = 0$ from the above inequality, we deduce that 
		\begin{equation}\label{eq_est_der_1}
			\dot{H}^1_{k, 0} \geq - T_k(h^L_0, h^L_1) \geq \dot{H}^2_{k, 0}.
		\end{equation}
		Directly from the definition of $H^1_{k, t}$, $H^2_{k, t}$, we deduce that 
		\begin{equation}\label{eq_est_der_2}
		\begin{aligned}
			& \dot{H}^1_{k, 0} = ({\textrm{Hilb}}_k(h^L_0, \chi_0))^{-1} \frac{d}{dt} {\textrm{Hilb}}_k(h^L_t, \chi_t)|_{t = 0} + C_1 g'(1) {\rm{Id}},
			\\
			& \dot{H}^2_{k, 0} = ({\textrm{Hilb}}_k(h^L_0, \chi_0))^{-1} \frac{d}{dt} {\textrm{Hilb}}_k(h^L_t, \chi_t)|_{t = 0} - C_1 g'(1) {\rm{Id}}.
		\end{aligned}
		\end{equation}
		From the definition of the $L^2$-norm, it is direct to see that
		\begin{equation}\label{eq_est_der_3}
			({\textrm{Hilb}}_k(h^L_0, \chi_0))^{-1} \frac{d}{dt} {\textrm{Hilb}}_k(h^L_t, \chi_t)|_{t = 0}
			=
			-
			k T_{\phi(h^L_0, h^L_1), k}
			+
			T_{  \chi_0^{-1} \frac{d}{dt} \chi_t|_{t = 0}, k}.
		\end{equation}
		From (\ref{eq_est_der_1}), (\ref{eq_est_der_2}) and (\ref{eq_est_der_3}), the result follows directly.
	\end{proof}
	\par 
	Now, to drop the additional assumption from Theorem \ref{thm_trasnfer_sm}, we need to regularize the Mabuchi geodesics.
	Let us recall the approximation scheme, called $\epsilon$-geodesics, introduced by X. Chen \cite{ChenGeodMab}.
	For this, we follow the notations introduced in (\ref{eq_ma_geod}).
	We fix $u_0, u_1 \in \mathcal{H}_{\omega}$, and for an arbitrary $\epsilon > 0$, consider the Dirichlet problem 
	\begin{equation}\label{eq_eps_geod}
		(\pi^* \omega + \imun \partial \dbar \hat{u})^{n + 1}
		=
		\epsilon \pi^* \omega^n \wedge \imun dz \wedge d \overline{z}, 
	\end{equation}
	with boundary conditions $\hat{u}(x, e^{\imun \theta}) = u_0(x)$, $\hat{u}(x, e^{-1 + \imun \theta}) = u_1(x)$, $x \in X, \theta \in [0, 2\pi]$.
	\begin{thm}[{\cite[Lemma 7]{ChenGeodMab}, \cite[\S 3]{ChuTossVeinC11}}]\label{thm_exist_eps_geod}
		For any $\epsilon > 0$, there is a unique $S^1$-invariant smooth solution $\hat{u}_{\epsilon}$ to (\ref{eq_eps_geod}), for which $\pi^* \omega + \imun \partial \dbar \hat{u}_{\epsilon}$ is strictly positive.
		Moreover, as $\epsilon \to 0$, $\hat{u}_{\epsilon}$ converges to the solution $\hat{u}$ of (\ref{eq_ma_geod}) in $\mathscr{C}^{1}(X \times \mathbb{D}(e^{-1}, 1))$.
		In particular, as $\epsilon \to 0$, $\dot{u}_{0, \epsilon} := \frac{d}{dt} u_{t, \epsilon}|_{t = 0}$ converges to $\dot{u}_0$ uniformly. 
	\end{thm}
	We will also need another ingredient from \cite[Theorem 4.20]{FinSecRing}, which we recall below.
	\begin{thm}\label{thm_equiv_geod}
		We fix continuous psh metrics $h^L_0$, $h^L_1$ on $L$, and a smooth path of Kähler forms $\chi_t$, $t \in [0, 1]$, on $X$. 
		There is a sequence $a_k \in \real$, verifying $a_k/k \to 0$, as $k \to \infty$, and $k_0 \in \nat$, such that for any $k \geq k_0$, $t \in [0, 1]$, the geodesic $H_{k, t}$, $t \in [0, 1]$, between ${\rm{Hilb}}_k(h^L_0, \chi_0)$ and ${\rm{Hilb}}_k(h^L_1, \chi_1)$ is related with the Mabuchi geodesic $h^L_t$, $t \in [0, 1]$, between $h^L_0$ and $h^L_1$ as follows
		\begin{equation}
			{\rm{Hilb}}_k(h^L_t, \chi_t)
			\cdot
			\exp(-a_k)
			\leq
			H_{k, t}
			\leq
			{\rm{Hilb}}_k(h^L_t, \chi_t)
			\cdot
			\exp(a_k).
		\end{equation}
	\end{thm}
	\begin{rem}
		Berndtsson in \cite{BerndtProb}, following a previous work of Phong-Sturm \cite{PhongSturm}, proved the uniform convergence of $FS(H_{k, t})^{\frac{1}{k}}$ to $h^L_t$, cf. also \cite{DarvLuRub} for a related result. 
		Theorem \ref{thm_equiv_geod} refines this result. 
		Remark, however, that our proof in \cite{FinSecRing} relies on the results of Berndtsson.
	\end{rem}
	\begin{proof}[Proof of Theorem \ref{thm_trasnfer}]
		Let us first reduce our considerations to the case when both endpoints are smooth.
		We use the notations introduced in Proposition \ref{prop_cont_der}.
		Then by Theorem \ref{thm_interpol}, we get
		\begin{equation}\label{eq_est_der_1_0}
			T_k(h^L_0, h^L_{1, \epsilon}) - \epsilon k {\rm{Id}}
			\leq
			T_k(h^L_0, h^L_1)
			\leq
			T_k(h^L_0, h^L_{1, \epsilon})
			+ \epsilon k {\rm{Id}}.
		\end{equation}
		If we now assume that Theorem \ref{thm_trasnfer} is valid for smooth endpoints, we deduce that for any $\epsilon > 0$, there is $k_0 \in \nat$, such that for any $k \geq k_0$, we have
		\begin{equation}\label{eq_est_der_3_0}
			\Big\|
				\frac{1}{k} T_k(h^L_0, h^L_{1, \epsilon})
				-
				T_{\phi(h^L_0, h^L_{1, \epsilon}), k}
			\Big\|
			\leq
			\epsilon.
		\end{equation}
		Directly from (\ref{eq_est_der_2_0}), (\ref{eq_est_der_1_0}) and (\ref{eq_est_der_3_0}), we establish that Theorem \ref{thm_trasnfer} then holds for continuous endpoint $h^L_1$.
		It is, hence, enough to establish Theorem \ref{thm_trasnfer} when both endpoints are smooth, which we assume from now on.
		\par 
		We denote by $h^L_{t, \epsilon}$, $t \in [0, 1]$, the $\epsilon$-geodesic between $h^L_0$ and $h^L_1$ associated with an auxillary form $\omega := 2 \pi c_1(L, h^L_0)$, given by Theorem \ref{thm_exist_eps_geod}.
		We denote by $\hat{H}_{k, \epsilon}$ the rotationally-invariant Hermitian metric on the (trivial) vector bundle $H^0(X, L^{\otimes k}) \times \mathbb{D}(e^{-1}, 1)$ over $\mathbb{D}(e^{-1}, 1)$, constructed from ${\textrm{Hilb}}_k(h^L_{t, \epsilon})$, $t \in [0, 1]$, as in (\ref{eq_defn_hat_H}).
		For $\tau \in \mathbb{D}(e^{-1}, 1)$, we define $d_{\epsilon}(\tau) : X \to \real$ as follows
		\begin{equation}
			d_{\epsilon}(\tau) := \frac{1}{n + 1} \frac{c_1(L, h^L_0)^n}{c_1(L, h^L_{t, \epsilon})^n}, \quad \text{where} \quad t = |\log |\tau||.
		\end{equation}
		Directly from (\ref{eq_eps_geod}), by Theorem \ref{thm_mazh}, we deduce that for any $\epsilon > 0$, there are $C > 0$, $k_0 \in \nat$, such that the curvature $R_{k, \epsilon}$ of $\hat{H}_{k, \epsilon}$ satisfies 
		\begin{equation}\label{eq_r_k_eps}
			\Big\| \frac{\imun}{2 \pi} R_{k, \epsilon, \tau} - \epsilon k  T_{d_{\epsilon}(\tau), k} \Big\| \leq C.
		\end{equation}
		\par 
		The major technical problem is that we do not know if $d_{\epsilon}(\tau)$ stays bounded uniformly in $\tau$, as $\epsilon \to 0$.
		In fact, it is even likely that it will not stay bounded as long as the Mabuchi geodesic passes through a non strictly positive Hermitian metric.
		The absence of this \textit{non-collapsing estimates} makes it unclear if (\ref{eq_r_k_eps}) implies that $\frac{1}{k} \| R_{k, \epsilon, \tau} \|$ can be made arbitrary small uniformly in $\tau \in \mathbb{D}(e^{-1}, 1)$, for $k \in \nat$ big enough and $\epsilon > 0$ small enough.
		This, however, was crucial in the construction of subinterpolating family of metrics as described in the proof of Theorem \ref{thm_trasnfer_sm}.
		The rest of the proof will be dedicated to a trick which will overcome this problem.
		\par 
		First, remark that since $d_{\epsilon}(\tau) > 0$, we deduce from (\ref{eq_r_k_eps}) that the curvature $R_{k, \epsilon, \tau}$ eventually becomes positive, for $k$ big enough.
		From Theorem \ref{thm_curv_geod}, we then deduce that
		\begin{equation}
			{\textrm{Hilb}}_k(h^L_{t, \epsilon}) \geq H_{t, k}.
		\end{equation}
		By proceeding in exactly the same way as we did in the proof of Theorem \ref{thm_trasnfer_sm}, we then see that for any $\epsilon > 0$, there are $k_0 \in \nat$, $C > 0$, such that for any $k \geq k_0$, we have
		\begin{equation}\label{eq_transfer_first}
			 k T_{\dot{h}^L_{0, \epsilon}, k}
			 + C \cdot {\rm{Id}}
			 \geq
			 - T_k(h^L_0, h^L_1).
		\end{equation}
		\par 
		We will now show how to modify the argument from Theorem \ref{thm_trasnfer_sm} to get the lower bound for $T_k(h^L_0, h^L_1)$.
		We fix a smooth path of Kähler forms $\chi_t$, $t \in [0, 1]$ between $\chi_0$ and $\chi_1$.
		By the $\mathscr{C}^1$-convergence of $\epsilon$-geodesics to the Mabuchi geodesic from Theorem \ref{thm_exist_eps_geod}, and Theorem \ref{thm_equiv_geod}, we conclude that there is a sequence $a_k \in \real$, $k \in \nat$, such that $a_k / k \to 0$, as $k \to \infty$, and so that for any $\delta > 0$, there are $k_0 \in \nat$, $\epsilon_0 > 0$, so that for any $t \in [0, 1]$, $k \geq k_0$, $\epsilon_0 > \epsilon > 0$, we have 
		\begin{equation}\label{eq_geod_c0_bnd}
			H_{t, k}
			\geq
			{\textrm{Hilb}}_k(h^L_{t, \epsilon}, \chi_t) \cdot \exp(- \delta t k - a_k).
		\end{equation}
		\par 
		Now, since $h^L_{t, \epsilon}$, $t \in [0, 1]$, is a smooth path, and $d_{\epsilon}(1) = \frac{1}{n + 1}$, there is $t_{\epsilon} > 0$, such that $d_{\epsilon}(\tau) \leq 1$ for all $|\tau| = e^{-t}$, $t \in [0, t_{\epsilon}]$.
		Then, in the notations (\ref{eq_g_func_defn}), we deduce that there is $C > 0$, such that for any $\epsilon > 0$, there is $k_0 \in \nat$, such that the curvature of the Hermitian vector bundle 
		$\hat{H}^2_{k, \epsilon} = \hat{H}_{k, \epsilon} \cdot \exp(\epsilon C k g)$ is negative over $X \times \mathbb{D}(e^{- t_{\epsilon}}, 1)$ for any $k \geq k_0$.
		Directly from Theorem \ref{thm_curv_geod}, we deduce that the geodesic $H_{t, k}^{0}$, $t \in [0, t_{\epsilon}]$, from ${\textrm{Hilb}}_k(h^L_0, \chi_0)$ to ${\textrm{Hilb}}_k(h^L_{t_{\epsilon}, \epsilon}, \chi_{t_{\epsilon}}) \cdot \exp(\epsilon C k g(e^{-t_{\epsilon}}))$, is related with the above path for any $k \geq k_0$, $t \in [0, t_{\epsilon}]$, as follows
		\begin{equation}\label{eq_geod_c0_bnd0}
			H_{t, k}^{0}
			\geq
			{\textrm{Hilb}}_k(h^L_{t, \epsilon}, \chi_t) \cdot \exp(\epsilon C k g(e^{-t})).
		\end{equation}
		Remark that we have $H_{0, k} = H_{0, k}^{0}$, and by the negativity of $g$, (\ref{eq_geod_c0_bnd}) and (\ref{eq_geod_c0_bnd0}), we also have $H_{t_{\epsilon}, k} \geq H_{t_{\epsilon}, k}^{0} \cdot \exp (-  t_{\epsilon} \delta k -  a_k)$.
		From this and Proposition \ref{thm_interpol}, we conclude that for any $t \in [0, t_{\epsilon}]$, $k \geq k_0$, we have
		\begin{equation}\label{eq_geod_c0_bnd2}
			H_{t, k}
			\geq
			H_{t, k}^{0} 
			\cdot 
			\exp (-  t \delta k -  t a_k / t_{\epsilon}).
		\end{equation}
		A combination of (\ref{eq_geod_c0_bnd0}) and (\ref{eq_geod_c0_bnd2}) yields that for any $t \in [0, t_{\epsilon}]$, $k \geq k_0$, we have
		\begin{equation}\label{eq_geod_c0_bnd3}
			H_{t, k}
			\geq
			{\textrm{Hilb}}_k(h^L_{t, \epsilon}, \chi_t) \cdot 
			\exp (\epsilon C k g(e^{-t}) -  t \delta k -  t a_k / t_{\epsilon}  ).
		\end{equation}
		If we then take a derivative at $t = 0$ of (\ref{eq_geod_c0_bnd3}), by (\ref{eq_est_der_3}), we immediately get that for any $\delta > 0$, there is $\epsilon_0 > 0$ so that for any $\epsilon_0 > \epsilon > 0$, there are $t_{\epsilon} > 0$, $k_0 \in \nat$, so that for any $k \geq k_0$, we have
		\begin{equation}\label{eq_transfer_second}
			- T_k(h^L_0, h^L_1)
			\geq
			k T_{\dot{h}^L_{0, \epsilon}, k}
			-
			\Big(
			\epsilon C k g'(1) + \delta k + a_k / t_{\epsilon}  - \inf \big( \chi_0^{-1} \frac{d}{dt} \chi_t|_{t = 0} \big)
			\Big)
			{\rm{Id}}.
		\end{equation}
		A combination of (\ref{eq_transfer_first}) and (\ref{eq_transfer_second}) finally yields the result by the uniform convergence of $\dot{h}^L_{0, \epsilon}$ towards $-\phi(h^L_0, h^L_1)$ from Theorem \ref{thm_exist_eps_geod}.
	\end{proof}

\section{Functional calculus for Toeplitz operators}\label{sect_func_calcul}
The primary objective of this section is to demonstrate that the convergence of operators in various norms leads to the convergence of their diagonal kernels, whereas the converse typically does not hold (unless the sequence of operators forms a Toeplitz operator).
As a consequence of our considerations, we show that the vector space of Toeplitz operators of Schatten class form an algebra closed by the functional calculus.
Through this, we demonstrate that Theorem \ref{thm_berg_conv} becomes a consequence of Theorems \ref{thm_main1}, \ref{thm_main2} if one replaces the pointwise convergence by the convergence in $L^p(X)$-spaces, for any $p \in [1, +\infty[$.
\par 
Throughout the whole section we fix a positive Hermitian metric $h^L$ on $L$, and a sequence $T_k \in {\rm{End}}(H^0(X, L^{\otimes k}))$ of Hermitian operators (with respect to ${\textrm{Hilb}}_k(h^L)$).
Recall that $N_k := \dim H^0(X, L^{\otimes k})$.
We denote the associated volume form $d \nu := \frac{1}{\int c_1(L)^n} c_1(L, h^L)^n$.
We use the notations $\| \cdot \|$, $\| \cdot \|_p$, $p \in [1, +\infty[$, for the norms introduced in Definition \ref{defn_toepl}, \ref{defn_toepl_sch}.
	\par 
	Let us first discuss the relation between the norm $\| \cdot \|$ and the diagonal kernel. 
	\begin{prop}\label{prop_c0_bnd}
		For any $x \in X$, we have $| T_k(x) | \leq B_k(x) \cdot \| T_k \|$.
		While this inequality is sharp, already for $(X, L) := (\mathbb{P}^1, \mathscr{O}(1))$, for any $C > 0$, there are $k \in \nat$, $T_k \in {\rm{End}}(H^0(X, L^{\otimes k}))$ so that $| T_k(x) | \leq B_k(x)$, for any $x \in X$, but $\| T_k \| \geq C$.
	\end{prop}
	\begin{proof}
		First of all, by the independence of the expression like in (\ref{eq_weight_berg}) on the choice of orthonormal basis, it is immediate that 
		\begin{equation}\label{eq_bergm_kern_ident}
			T_k(x)
			=
			\big\langle T_k s_{x, k}, s_{x, k} \big\rangle_{{\textrm{Hilb}}_k(h^L)} \cdot B_k(x),
		\end{equation}
		where $s_{x, k} \in H^0(X, L^{\otimes k})$ is a peak section at $x \in X$, which is a section of norm $1$ with respect to ${\textrm{Hilb}}_k(h^L)$, orthogonal to all sections from $H^0(X, L^{\otimes k})$ vanishing at $x$.
		This shows the first part of Proposition \ref{prop_c0_bnd} by the trivial inequality $\langle T_k s_{x, k}, s_{x, k} \rangle_{{\textrm{Hilb}}_k(h^L)} \leq \| T_k \|$.
		The sharpness of the established inequality is trivial for $T_k := B_k$.
		\par 
		For the second part of the statement, we assume $(X, L) := (\mathbb{P}^1, \mathscr{O}(1))$ is endowed with the Fubini-Study metric as described in the example in the end of Section \ref{sect_prel}, from which we also borrow the notation.
		We consider an operator $T_{2k} := \frac{1}{2} ( B_{2k} + \sqrt{k} \cdot P_{2k})$, where $P_{2k}$ is the orthogonal projection from $H^0(X, L^{\otimes 2k})$ to the subspace spanned by $x^k \cdot y^k$ in the notations from (\ref{eq_l2_norm_calc}).
		Clearly, $\| T_{2k} \| = \frac{ \sqrt{k} + 1}{2}$.
		Directly from (\ref{eq_l2_norm_calc}), we see that 
		\begin{equation}
			T_{2 k}([a x^* + b y^*])
			=
			\frac{2k + 1}{2}
			+ 
			\frac{\sqrt{k}}{2}
			\Big( \frac{|a| |b|}{|a|^2 + |b|^2} \Big)^{2k} \frac{(2 k + 1)!}{k! k!}.
		\end{equation}
		Directly from Stirling's approximation and inequality $|a|^2 + |b|^2 \geq 2 |a| |b|$, we deduce that for any $x \in \mathbb{P}^1$, $|T_{2k}(x)| \leq 2k + 1$, finishing the proof.
	\end{proof}
	Let us now discuss the relation between the norms $\| \cdot \|_p$, $p \in [1, +\infty[$, and the diagonal kernel. 
	\begin{prop}\label{prop_l1_bnd}
		For any $\epsilon > 0$, there is $k_0 \in \nat$, such that for any $k \geq k_0$, $p \in [1, +\infty[$, we have 
		\begin{equation}\label{eq_l1_bnd}
			\sqrt[p]{ \int |T_k(x)|^p d \nu(x) } \leq \| T_k \|_p \cdot k^n \cdot (1 + \epsilon).
		\end{equation}
		While this inequality is sharp, already for $(X, L) := (\mathbb{P}^1, \mathscr{O}(1))$, there are $T_k \in {\rm{End}}(H^0(X, L^{\otimes k}))$, $k \in \nat$, so that $\| T_k \| = \| T_k \|_p = 1$, for any $p \in [1, +\infty[$, but for any $\epsilon > 0, p \in [1, +\infty[$, there is $k_0 \in \nat$, such that for any $k \geq k_0$, we have $\int |T_k(x)|^p d \nu(x) \leq \epsilon \cdot k^{np}$.
	\end{prop}
	\begin{proof}
		Let us first establish the inequality 
		\begin{equation}\label{eq_tk_l1_nm}
			\int |T_k(x)| d \nu(x) \cdot \frac{\int c_1(L)^n}{n!} \leq \| T_k \|_1 \cdot N_k.
		\end{equation}
		If $T_k$ is positive-definite, (\ref{eq_tk_l1_nm}) immediately follows, as we have $\int T_k(x) d \nu(x) \cdot \int c_1(L)^n / n! = {\rm{Tr}}(T_k) = \| T_k \|_1 \cdot N_k$. 
		By decomposing $T_k$ as $T_k := T_k^+ - T_k^-$, where $T_k^+$, $T_k^-$ is a positive and negative parts of $T_k$ respectively, it is then immediate that (\ref{eq_tk_l1_nm}) holds in full generality.
		\par 
		From (\ref{eq_thm_tian}) and the asymptotic Riemann-Roch-Hirzebruch theorem, saying that $N_k \sim k^n \cdot \int c_1(L)^n / n!$, we see that (\ref{eq_tk_l1_nm}) refines (\ref{eq_l1_bnd}) for $p = 1$. 
		Moreover, by Proposition \ref{prop_c0_bnd}, we see that (\ref{eq_l1_bnd}) also holds for $p \to +\infty$.
		We establish the general case by interpolation.
		\par 
		We consider a map $\pi_k : {\rm{End}}(H^0(X, L^{\otimes k})) \to L^{\infty}(X)$, $T_k \mapsto (x \mapsto T_k(x))$.
		For $p \in [1, +\infty[$, we denote by $\| \pi_k \|_p$ the norm of $\pi_k$, viewed as an operator from $({\rm{End}}(H^0(X, L^{\otimes k})), \| \cdot \|_p)$ to $(L^{\infty}(X), \| \cdot \|_{L^p(X)})$.
		Clearly, (\ref{eq_l1_bnd}) just says that for any $\epsilon > 0$, there is $k_0 \in \nat$, so that for any $k \geq k_0$, $p \in [1, +\infty[$, we have $\| \pi_k \|_p \leq k^n \cdot (1 + \epsilon)$.
		Recall that $({\rm{End}}(H^0(X, L^{\otimes k})), \| \cdot \|_p)$ is a complex interpolation between $({\rm{End}}(H^0(X, L^{\otimes k})), \| \cdot \|_1)$ and $({\rm{End}}(H^0(X, L^{\otimes k})), \| \cdot \|)$, cf. \cite[Theorem 13.1]{GohbergKrein}.
		Similarly, $(L^{\infty}(X), \| \cdot \|_{L^p(X)})$ is a complex interpolation between $(L^{\infty}(X), \| \cdot \|_{L^1(X)})$ and $(L^{\infty}(X), \| \cdot \|_{L^{\infty}(X)})$, cf. \cite[Theorem 5.1.1]{InterpSp}.
		The result now follows directly from the fact that complex interpolation is an example of an \textit{exact interpolation functor}, see \cite[Theorem 4.1.2]{InterpSp} for the proof of this result and \cite[(6) on p.27]{InterpSp} for the necessary definitions.
		The sharpness of (\ref{eq_l1_bnd}) is again easily seen for $T_k := B_k$ by (\ref{eq_thm_tian}).
		\par 
		For the second part of the statement, we assume $(X, L) := (\mathbb{P}^1, \mathscr{O}(1))$ is endowed with the Fubini-Study metric as described in the example in the end of Section \ref{sect_prel}, from which we also borrow the notation.
		We consider an operator $T_k$ which is diagonal in the monomial basis, and which sends $x^i \cdot y^j$ to $(-1)^i x^i \cdot y^j$ in the notations (\ref{eq_l2_norm_calc}). 
		Clearly, we have $\| T_k \| = \| T_k \|_p = 1$ for any $p \in [1, +\infty[$, $k \in \nat$.
		Directly from (\ref{eq_l2_norm_calc}) and binomial identity, we see that 
		\begin{equation}
			T_k([a x^* + b y^*])
			=
			(k + 1) \cdot 
			\Big(\frac{|b|^2 - |a|^2}{|a|^2 + |b|^2}\Big)^k.
		\end{equation}
		It is immediate to see that for any $p \in [1, +\infty[$, we have $\lim_{k \to \infty} \int \big| \frac{1 - |z|^2}{1 + |z|^2} \big|^{pk} \frac{\imun dz d \overline{z}}{1 + |z|^2} = 0$, which implies that $T_k$ verifies the assumptions from Proposition \ref{prop_l1_bnd}.
	\end{proof}
	\par 
	Let us now discuss the diagonal kernels for Toeplitz operators and show that the subtleties pointed out in the second parts of Propositions \ref{prop_c0_bnd}, \ref{prop_l1_bnd} disappear for these operators.
	More precisely, we will establish the following two results.
	\begin{thm}\label{thm_toepl_refines}
		For any Toeplitz operator  (resp. of Schatten class) $\{ T_k \}_{k = 0}^{+\infty}$ with symbol $f \in \mathscr{C}^0(X)$ (resp. $f \in L^{\infty}(X)$), the sequence of functions $x \mapsto \frac{1}{k^n} T_k(x)$, $x \in X$, is uniformly bounded and converges uniformly (resp. in $L^p(X)$ for every $p \in [1, +\infty[$) to $f$.
	\end{thm}
	\begin{thm}\label{thm_toepl_norm}
		For any Toeplitz operator of Schatten class $\{ T_k \}_{k = 0}^{+\infty}$, with symbol $f \in L^{\infty}(X)$, for any $p \in [1, +\infty[$, we have
		\begin{equation}
			\lim_{k \to \infty} \| T_k \|_p
			=
			\sqrt[p]{  \int |f(x)|^p d \nu(x) }
			=
			\lim_{k \to \infty}
			\frac{1}{k^n}
			\sqrt[p]{  \int |T_k(x)|^p d \nu(x) }.
		\end{equation}
		Similarly, for any Toeplitz operator $\{ T_k \}_{k = 0}^{+\infty}$, with symbol $f$, we have $\lim_{k \to \infty} \| T_k \| = {\rm{sup}} |f(x)| = \lim_{k \to \infty} \frac{1}{k^n} {\rm{sup}} | T_k(x) |$.
	\end{thm}
	In order to prove these statements, we will need several auxiliary results. 
	\begin{lem}\label{lem_conv_kernel}
		For any $p \in [1, +\infty[$, $f \in L^p(X)$, the sequence of functions $k \in \nat$, $x \mapsto \frac{1}{k^n} T_{f, k}(x)$, $x \in X$, converges to $f$ in $L^p(X)$, as $k \to \infty$.
		If, moreover, $f$ is continuous, then the convergence is uniform.
	\end{lem}
	\begin{rem}\label{rem_conv_kernel}
		Ma-Marinescu in \cite[Theorem 0.1]{MaMarBTKah} established the analogous result in realms of Toeplitz operators associated with smooth symbols $f$.
		For continuous $f$, the result was previously established \cite[Theorem 3.3]{BarrMa}.
	\end{rem}
	\begin{proof}
		By Remark \ref{rem_conv_kernel}, it suffices to show the $L^p(X)$-convergence.
		Remark that $T_{f, k}(x)$ can be expressed in terms of the Bergman kernel $B_k(x, y) \in L_x^{\otimes k} \otimes (L_y^{\otimes k})^*$, $x, y \in X$, as follows
		\begin{equation}\label{eq_toepl_diag}
			T_{f, k}(x)
			=
			\int f(y) |B_k(x, y)|^2 d \nu(y) \cdot \frac{\int c_1(L)^n}{n!}.
		\end{equation}
		Now, for $x \in X$, we denote by $\exp_x: T_x X \to X$ the geodesic coordinates, considered with respect to the Kähler form $c_1(L, h^L)$.
		The main results from Dai-Liu-Ma \cite{DaiLiuMa} and Ma-Marinescu \cite{MaMarOffDiag}, cf. \cite[Theorem 4.2.1]{MaHol}, imply that there are $c, C, \epsilon > 0$, $k_0 \in \nat$, such that for any $k \geq k_0$, $x \in X$, $Z \in T_x X$, $|Z| < \epsilon$, we have 
		\begin{equation}\label{eq_dai_liu_ma}
		\begin{aligned}
			& \Big|
			|B_k(x, \exp_x(Z))|
			-
			k^n 
			\exp \big(- \frac{\pi}{2} k |Z|^2 \big)
			\Big|
			\leq 
			C
			k^{n - \frac{1}{2}} \cdot (\sqrt{k} |Z|)^{2 n + 1} \exp(- c \sqrt{k} |Z|),
			\\
			&
			|B_k(x, y)|
			\leq
			C k^n \exp( - c \sqrt{k} {\rm{dist}}(x, y)).
		\end{aligned}	
		\end{equation}
		Remark that for any $c > 0$, there is $C_0 > 0$, such that for any $k \in \nat^*$, we have
		\begin{equation}\label{eq_int_gauss}
			\int_{Z \in \comp^n} k^{n} \cdot (\sqrt{k} |Z|)^{2 n + 1} \exp(- c \sqrt{k} |Z|) dZ
			\leq 
			C_0.
		\end{equation}
		In particular, if we define the function $G_{f, k} : X \to \real$ as follows
		\begin{equation}
			G_{f, k}(x) = 
			k^n
			\int_{Z \in T_x X} f(\exp_x(Z)) \exp \big(- \pi k |Z|^2 \big) d Z,
		\end{equation}
		then directly from Generalized Young's Inequality, cf. \cite[(0.10)]{FollandBook}, (\ref{eq_dai_liu_ma}) and (\ref{eq_int_gauss}), we deduce that there is $C_1 > 0$, such that for any $k \geq k_0$, $f \in L^p(X)$, we have
		\begin{equation}\label{eq_error_term}
			\Big(
			\int
			\Big|
				\frac{1}{k^n} T_{f, k}(x) - G_{f, k}(x)
			\Big|^p d \nu(x)
			\Big)^{\frac{1}{p}}
			\leq
			\frac{C_1}{\sqrt{k}}
			\Big(
			\int
			\big|
				f(x)
			\big|^p d \nu(x)
			\Big)^{\frac{1}{p}}.
		\end{equation}
		Remark, however, that by the usual properties of convolutions, cf. \cite[Theorem (0.13)]{FollandBook}, the sequence of functions $x \mapsto G_{f, k}(x)$ converges to $f$ in $L^p(X)$, as $k \to \infty$.
		The result now follows from this and (\ref{eq_error_term}).
	\end{proof}
	
	\begin{lem}\label{lem_l1_bnd}
		For any $\epsilon > 0$, there is $k_0 \in \nat$, such that for any $f \in L^1(X)$, $k \geq k_0$, we have
		\begin{equation}
			\| T_{f, k} \|_1 \leq (1 + \epsilon) \int |f(x)| d \nu(x).
		\end{equation}
		For any $f \in L^{\infty}(X)$, $k \geq k_0$, we have $\| T_{f, k} \| \leq {\rm{esssup}}_{x \in X} |f(x)|$, for any $k \in \nat$.
	\end{lem}
	\begin{proof}
		Remark first that if $f$ is positive, then $\| T_{f, k} \|_1 = \frac{1}{N_k} \int 
			T_{f, k}(x)
			d \nu(x) \cdot \int c_1(L)^n / n!$.
		As a direct consequence of (\ref{eq_toepl_diag}) and Tonelli's theorem, we deduce
		\begin{equation}\label{eq_tonelli}
			\int 
			T_{f, k}(x)
			d \nu(x)
			=
			\int f(y) |B_k(x, y)|^2 d \nu(x) d \nu(y)  \cdot \frac{\int c_1(L)^n}{n!}.
		\end{equation}
		Remark, however, that $\int |B_k(x, y)|^2 d \nu(y)  \cdot \frac{\int c_1(L)^n}{n!} = B_k(x)$, and so we have
		\begin{equation}\label{eq_tonelli12212}
			\int 
			| T_{f, k}(x) |
			d \nu(x)
			\leq
			\int |f(x)| B_k(x) d \nu(x).
		\end{equation}
		The first part of Lemma \ref{lem_l1_bnd} for positive $f$ now follows from (\ref{eq_thm_tian}), (\ref{eq_tonelli12212}) and the asymptotic Riemann-Roch-Hirzebruch theorem. 
		For general $f$, The first part of Lemma \ref{lem_l1_bnd} is established by decomposing $f$ as a difference of a positive and a negative part and applying Lemma \ref{lem_l1_bnd} for each of them. The second part follows directly from the trivial bounds ${\rm{essinf}}_{x \in X} f(x) \cdot {\rm{Id}} \leq T_{f, k} \leq {\rm{esssup}}_{x \in X} f(x) \cdot {\rm{Id}}$.
	\end{proof}
	\begin{proof}[Proof of Proposition \ref{prop_conv_toepl_symb}]
		From the uniform bound assumption, the second part of Lemma \ref{lem_l1_bnd} and (\ref{eq_a_k_bnd_p1}), we see that it suffices to establish that for any $\epsilon > 0$, there is $k_0 \in \nat$, such that for any $k \geq k_0$, we have
		\begin{equation}
			\| T_{f_k, k} - T_{f, k} \|_1 \leq \epsilon.
		\end{equation}
		This follows directly from the first part of Lemma \ref{lem_l1_bnd} and the assumption on the $L^1(X)$-convergence of $f_k$ towards $f$.
	\end{proof}
	\begin{proof}[Proof of Theorem \ref{thm_toepl_refines}]
		It follows immediately from Propositions \ref{prop_c0_bnd}, \ref{prop_l1_bnd} and Lemma \ref{lem_l1_bnd}.
	\end{proof}
	
	\par 
	We will now study the functional-analytic and algebraic properties of the vector space of Toeplitz operators (resp. of Schatten class).
	\begin{prop}\label{prop_prod_toepl}
		For any continuous $f, g : X \to \real$ (resp. $f, g \in L^{\infty}(X)$) the sequence of operators $T_{f, k} \circ T_{g, k}$ is a Toeplitz operator (resp. of Schatten class) with symbol $f g$.
	\end{prop}
	\begin{proof}
		When both $f$, $g$ are smooth, the result is well known, see \cite{BordMeinSchli}, \cite[\S 7]{MaHol}.
		The general case follows easily by approximation and Lemma \ref{lem_l1_bnd}.
		For brevity, we only give the main idea.
		We fix $C := {\rm{esssup}} \max (|f|, |g|)$ and for any $\epsilon > 0$, consider smooth $f_{\epsilon}$, $g_{\epsilon}$, verifying $\sup \max(|f_{\epsilon}|, |g_{\epsilon}|) \leq C + 1$ and $\int |f(x) - f_{\epsilon}(x)| d \eta(x) < \epsilon \int d \eta(x)$, $\int |g(x) - g_{\epsilon}(x)| d \eta(x) < \epsilon \int d \eta(x)$, the existence of which follows from the usual density statements.
		It is then a direct that for any $\epsilon > 0$, there is $k_0 \in \nat$, such that for any $k \geq k_0$, we have 
		\begin{equation}\label{eq_toepl_prod_reg}
			\Big\| T_{f, k} \circ T_{g, k} - T_{f_{\epsilon}, k} \circ T_{g_{\epsilon}, k} \Big\|
			\leq
			4 \epsilon (C + 1), 
			\qquad
			\Big\| T_{f_{\epsilon} g_{\epsilon}, k} - T_{f g, k} \Big\|
			\leq
			4 \epsilon (C + 1).
		\end{equation}
		From Lemma \ref{lem_l1_bnd}, (\ref{eq_toepl_prod_reg}) and the validity of Proposition \ref{prop_prod_toepl} for $f_{\epsilon}$, $g_{\epsilon}$, we conclude that for any $\epsilon > 0$, there is $k_1 \in \nat$, such that for any $k \geq k_1$, we have
		\begin{equation}\label{eq_toepl_prod_reg1}
			\Big\| T_{f_{\epsilon}, k} \circ T_{g_{\epsilon}, k} - T_{f_{\epsilon} g_{\epsilon}, k} \Big\|
			\leq
			5 \epsilon (C + 1).
		\end{equation}
		A combination of (\ref{eq_toepl_prod_reg}) and (\ref{eq_toepl_prod_reg1}) yields the proof.
	\end{proof}
	\begin{cor}
		The vector space of Toeplitz operators (resp. of Schatten class) forms an algebra, and the symbol map is a algebra morphism.
	\end{cor}
	\begin{proof}
		Follows directly from Proposition \ref{prop_prod_toepl}.
	\end{proof}
	\begin{prop}\label{prop_func_toepl}
		For any continuous functions $g : \real \to \real$, $f : X \to \real$ (resp. $f \in L^{\infty}(X)$), the sequence of operators $\{ g(T_{f, k}) \}_{k = 0}^{+\infty}$ form a Toeplitz operator (resp. of Schatten class) with symbol $g(f)$.
	\end{prop}
	\begin{proof}
		For polynomials $g$, the result follows from Proposition \ref{prop_prod_toepl}.
		The general case is done by approximation.
		The details are left to the reader.
	\end{proof}
	\begin{cor}\label{cor_funct_calc}
		The vector space of Toeplitz operators (resp. of Schatten class) is closed under functional calculus associated with continuous functions.
	\end{cor}
	\begin{proof}
		Follows directly from Proposition \ref{prop_func_toepl}.
	\end{proof}
	\par 
	\begin{sloppypar}
	\begin{proof}[Proof of Theorem \ref{thm_berg_conv} for $L^p(X)$-convergence in place of pointwise convergence] 
		By Proposition \ref{prop_func_toepl}, Theorem \ref{thm_main1} directly implies the second part of Theorem \ref{thm_berg_conv}. 
		Similarly, Proposition \ref{prop_func_toepl} shows that Theorem \ref{thm_main2} implies the first part of Theorem \ref{thm_berg_conv} when pointwise convergence is replaced by the convergence in $L^p(X)$-spaces for any $p \in [1, +\infty[$. 
	\end{proof}
	\end{sloppypar}	
	\begin{proof}[Proof of Theorem \ref{thm_toepl_norm}]
		Let us first establish the first part. 
		The second identity follows from Theorem \ref{thm_toepl_refines}.
		To establish the first identity, we deduce from Proposition \ref{prop_func_toepl} that $|T_k|^p$, $k \in \nat$, is a Toeplitz operator of Schatten class with symbol $|f|^p$.
		Then for any $\epsilon > 0$, there is $k_0 \in \nat$, such that we have 
		\begin{equation}
			\Big|
			{\rm{Tr}} \big[ | T_k |^p \big] 
			-
			{\rm{Tr}} \big[ T_{|f|^p, k} \big] 
			\Big|
			\leq
			\epsilon N_k.
		\end{equation}
		However, ${\rm{Tr}}[ T_{|f|^p, k}] = \int T_{|f|^p, k}(x) d \nu(x)  \cdot \int c_1(L)^n / n!$, and the first identity again follows from Proposition \ref{prop_func_toepl}.
		The second part of the statement is well known and follows directly from Lemma \ref{lem_l1_bnd} and Theorem \ref{thm_toepl_refines}.
	\end{proof}
	
	As a conclusion, we point out that the analogue of the spectral convergence result of Boutet de Monvel-Guillemin \cite{BoutGuillSpecToepl}, also holds in the more general context of Toeplitz operators of Schatten class.
	More specifically, let us establish the following result.
	\begin{prop}
		For any Toeplitz operator of Schatten class $\{ T_k \}_{k = 0}^{+\infty}$, with symbol $f \in L^{\infty}(X)$, the spectral measures of $T_k$ converge weakly to the probability measure $f_*(\nu)$ on $\real$.
	\end{prop}
	\begin{proof}
		It suffices to prove that for any continuous $g : \real \to \real$, we have the following convergence
		\begin{equation}
			\lim_{k \to \infty} \frac{1}{N_k} \sum_{\lambda \in {\rm{Spec}}(T_k)} g(\lambda)
			=
			\int_{x \in \real} g(f(x)) d \nu(x).
		\end{equation}
		However, we have $\sum_{\lambda \in {\rm{Spec}}(T_k)} g(\lambda) = {\rm{Tr}}[g(T_k)]$, and the result follows directly from Theorem \ref{thm_toepl_refines} and Corollary \ref{cor_funct_calc}.
	\end{proof}
	
	\section{Superadditivity of weighted Bergman kernels, a proof of Theorem \ref{thm_berg_conv}}\label{sect_subadd}
	The main goal of this section is to complete the proof of Theorem \ref{thm_berg_conv} by establishing the pointwise convergence. 
	To clarify our approach, let $x \mapsto B_k^{\mathcal{F}}(x)$ denote the weighted Bergman kernel, as defined in (\ref{eq_weight_berg}) for the function $g(x) := x$. 
	The main result of this section, stated below, demonstrates the partial superadditivity of weighted Bergman kernels.
	\begin{thm}\label{thm_subadd_berg}
		There are $k_0 \in \nat$, $C > 0$, so that for any $k, l \geq k_0$, $l / 2 \leq k \leq 2 l$, $x \in X$, we have
		\begin{equation}
			\frac{B_{k + l}^{\mathcal{F}}(x)}{(k + l)^{n - 1}}
			\geq
			\frac{B_{k}^{\mathcal{F}}(x)}{k^{n - 1}}
			+
			\frac{B_{l}^{\mathcal{F}}(x)}{l^{n - 1}}
			-
			C \log(k + l).
		\end{equation}
	\end{thm}
	The proof of Theorem \ref{thm_subadd_berg} relies heavily on the multiplicativity properties of $L^2$-metrics established by the author in \cite{FinSecRing}, and it will be presented in the end of the section. 
	Before that, let us see how it implies Theorem \ref{thm_berg_conv}.
	\begin{proof}[Proof of Theorem \ref{thm_berg_conv}]
		By the results of Section \ref{sect_func_calcul}, it suffices to establish the pointwise convergence of $x \mapsto \frac{1}{k^n} B_k^{\mathcal{F}}(x)$, $x \in X$, as $k \to \infty$.
		We will establish it by reducing to the case $g(x) := x$ and using Theorem \ref{thm_subadd_berg}.
		\par 
		First, recall that the classical Fekete's Subadditive Lemma says that for any sequence $a_k \in \real$, which is subadditive, i.e. for any $k, l \in \nat$, we have $a_{k + l} \leq a_k + a_l$, the limit of $\frac{1}{k} a_k$ exists in $\real \cup \{ - \infty\}$.
		De Bruijn-Erdős further generalized this result in \cite[Theorem 23]{BruijErd} to sequences, verifying $a_{k + l} \leq a_k + a_l + g_{k + l}$, for any $k, l \in \nat$, verifying $l/2 \leq k \leq 2 l$, and an auxillary sequence $g_k \in \real$, so that $\sum \frac{g_k}{k^2} < + \infty$.
		\par 
		Remark that $\sum \frac{\log(k)}{k^2} < + \infty$, and so the result of de Bruijn-Erdős along with Theorem \ref{thm_subadd_berg} imply the pointwise convergence of $x \mapsto \frac{1}{k^n} B_k^{\mathcal{F}}(x)$, $x \in X$, as $k \to \infty$.
		The boundness condition on $\mathcal{F}$ implies that the limit is bounded. 
		This establishes Theorem \ref{thm_berg_conv} for $g(x) := x$.
		\par 
		Let us show that it automatically implies Theorem \ref{thm_berg_conv} for $g(x) := g_c(x) := \min(x, c)$, for any $c \in \real$.
		In order to see this, for an arbitrary filtration $\mathcal{F}$, we define the filtration $\mathcal{F}^c$ through its weight function as $w_{\mathcal{F}^c}(s) := \min (w_{ \mathcal{F} }(s) , ck)$, $s \in H^0(X, L^{\otimes k})$, $k \in \nat$. 
		It is easy to see that $\mathcal{F}^c$ is submultiplicative and bounded whenever $\mathcal{F}$ is.
		Moreover, it is direct that in the notations of (\ref{eq_weight_berg}), we have $g_c ( A(\mathcal{F}_k, {\textrm{Hilb}}_k(h^L)) / k ) = A(\mathcal{F}^c_k, {\textrm{Hilb}}_k(h^L)) / k$.
		This in particular implies that $B_k^{\mathcal{F}^c}(x) = B_k^{\mathcal{F}, g_c}(x)$.
		And consequently the already established pointwise convergence result for $\mathcal{F}^c$, $g(x) := x$, implies the pointwise convergence for $\mathcal{F}$, $g(x) := g_c(x)$.
		\par 
		We claim that the real vector space spanned by the functions $g_c$, $c \in \real$, is dense within the set of continuous functions (with respect to the uniform topology over a fixed compact interval in $\real$). 
		This wold complete the proof of Theorem \ref{thm_berg_conv}, as it is straightforward to see that it suffices to verify the theorem's statement for a subset of continuous functions that is dense in the uniform topology among continuous functions with compact support in the interval $[-C, C]$, where $C > 0$ is such that $\max_{s \in H^0(X, L^{\otimes k}) \setminus \{0\}} |w_{\mathcal{F}}(s)| \leq C k$, for any $k \in \nat$.
		\par 
		To see the density of $g_c$, $c \in \real$, it suffices to establish that an arbitrary smooth function $g : \real \to \real$ of compact support lies in the closure of this vector space. 
		But this follows directly from the formula $g(x) = \int_{- \infty}^{+\infty} (x + t - g_t(x)) g''(t) dt$, the verification of which is left to the reader.
	\end{proof}	 
	\begin{rem}
		Our choice of a dense subset of functions was inspired by Chen-MacLean \cite[Proposition 5.1]{ChenMaclean}, who used such set in a related context.
	\end{rem}
	\par 
	From now on, we concentrate on the proof of Theorem \ref{thm_subadd_berg}.
	Recall that a Hermitian norm $H_V = \| \cdot \|_V$ on a finitely dimensional vector space $V$ naturally induces the Hermitian norm $\| \cdot \|_Q$, which we also denote by $[ H_V ]$, on any quotient $Q$, $\pi : V \to Q$ of $V$ through the following identity
	\begin{equation}\label{eq_defn_quot_norm}
		\| f \|_Q
		:=
		\inf \big \{
		 \| g \|_V
		 ;
		 \quad
		 g \in V, 
		 \pi(g) = f
		\},
		\qquad f \in Q.
	\end{equation}
	A similar construction associates for an arbitrary filtration $\mathcal{F}$ on $V$ the induced quotient filtration $[\mathcal{F}]$ on $Q$.
	Similarly, one can naturally identify $V$ with $Q \oplus \ker \pi$ using the dual to $\pi$ map $\pi^* : Q \to V$.
	Using this identification, for any $A \in \enmr{V}$, we then can define the operator $A|_Q \in \enmr{Q}$ by $A|_Q (q) = \pi(A( \pi^* (q)))$.
	\par 
	\begin{prop}[{cf. \cite[Proposition 4.12]{FinTits} }]\label{prop_stein_weiss}
		We denote by $H^V_t$, $t \in [0, +\infty[$, the geodesic ray of Hermitian metrics on $V$ associated with $\mathcal{F}$, departing from $H^V$.
		Similarly, we let $[H^V]_t$, $t \in [0, +\infty[$, be the geodesic ray of Hermitian metrics on $Q$ associated with $[\mathcal{F}]$, departing from $[H^V]$.
		Then for any $t \in [0, + \infty[$, we have
		\begin{equation}
			[H^V_t]
			\geq
			[H^V]_t.
		\end{equation}
		In particular, by taking the derivative at $t = 0$, we obtain
		\begin{equation}
			A(H^V, \mathscr{F})|_Q
			\leq
			A([H^V], [\mathscr{F}]).
		\end{equation}
	\end{prop}
	\par 
	The above result will be applied in the setting of a section ring.
	More specifically, for any $k, l \in \nat$, we denote the multiplication map
	\begin{equation}\label{eq_mult_map}
		{\rm{Mult}}_{k, l} : H^0(X, L^{\otimes k}) \otimes H^0(X, L^{\otimes l})
		\to
		H^0(X, L^{\otimes (k + l)}), \qquad s_1 \otimes s_2 \mapsto s_1 \cdot s_2.
	\end{equation}
	It is classical that there is $k_0 \in \nat$, such that ${\rm{Mult}}_{k, l}$ is surjective for any $k, l \geq k_0$.
	The following statement shows that $L^2$-norms respect the algebraic structure of $R(X, L)$ in a certain sense.
	\begin{thm}[{\cite[Theorem 1.1]{FinSecRing}}]\label{thm_as_isom}
		There are $C > 0$, $k_1 \in \nat^*$, such that for any $k, l \geq k_1$, for the norms over $H^0(X, L^{k + l})$, the following relation holds
		\begin{equation}\label{eq_as_isom1}
			1 - C \Big( \frac{1}{k} + \frac{1}{l} \Big)
			\leq 
			\frac{[{\rm{Hilb}}_k(h^L) \otimes {\rm{Hilb}}_l(h^L)]}{{\rm{Hilb}}_{k + l}(h^L)} 
			\cdot
			\Big( \frac{k \cdot l}{k + l} \Big)^{\frac{n}{2}}  
			\leq 
			1 + C \Big( \frac{1}{k} + \frac{1}{l} \Big),
		\end{equation}
		Where the quotient norm $[{\rm{Hilb}}_k(h^L) \otimes {\rm{Hilb}}_l(h^L)]$ is constructed using (\ref{eq_mult_map}).
	\end{thm}
	\par 
	Another preliminary result we shall use concerns the stability estimates for the weight operator, which compares the weight operators of a given filtration for two different Hermitian metrics.
	\begin{thm}\label{thm_cholesky}
		Assume that for a constant $C > 0$, verifying $(1 + 2 \lceil \log_2 \dim V \rceil)^2 C < 1$, the Hermitian products $H_0$, $H_1$ on $V$ satisfy the bound 
		\begin{equation}\label{eqthm_cholesky}
			1 - C
			\leq 
			\frac{H_1}{H_0}
			\leq
			1 + C.
		\end{equation}
		Then for any filtration $\mathcal{F}$ on $V$, the following bound is satisfied 
		\begin{equation}
			\Big\| 
				A(\mathcal{F}, H_0) - A(\mathcal{F}, H_1)
			\Big\|
			\leq
			16  C  (1 + 2 \lceil \log_2 \dim V \rceil) \| \mathcal{F} \|,
		\end{equation}
		where $\| \cdot \|$ is the operator norm subordinate with $H_0$, and $\| \mathcal{F} \| := \sup_{v \in V \setminus \{0\}} |w_{\mathcal{F}}(v)|$.
	\end{thm}
	\begin{proof}
		Denote $r := \dim V$ and fix a basis $f_1, \ldots, f_r$ of $V$, adapted to $\mathcal{F}$ and $H_0$.
		Let $T$ be a transfer map between $H_0$ and $H_1$, which we view as an $r \times r$ matrix using the above basis.
		Consider the \textit{Cholesky decomposition} of $\exp(-T)$, i.e. let $L$ be the lower triangular matrix, verifying $\exp(-T) = L L^*$.
		We claim that 
		\begin{equation}\label{eq_weight_chol}
			A(\mathcal{F}, H_1) = L^{-1} A(\mathcal{F}, H_0) L^*.
		\end{equation}
		Clearly, (\ref{eq_weight_chol}) is equivalent to the statement that $(L^*)^{-1} f_1, \ldots, (L^*)^{-1} f_r$ is a basis adapted to $\mathcal{F}$ and $H_0$.
		In order to show this, it suffices to verify that $(L^*)^{-1} f_i \in \mathcal{F}^{w_{\mathcal{F}}(f_i)}$ for $i = 1, \ldots, r$, and that $(L^*)^{-1} f_i$ form an orthonormal basis with respect to $H_1$.
		The first claim follows immediately from the fact that $(L^*)^{-1}$ is an upper triangular matrix with positive diagonal entries (the last fact follows from the fact that $\exp(-T)$ is positive definite, and so has positive diagonal entries).
		The second follows immediately from the definition of $L$ and the transfer map.
		\par 
		Now, recall that Cholesky decomposition is stable with respect to perturbations.
		More precisely, we denote $B := \exp(-T) - {\rm{Id}}$.
		Define the matrix $P(B)$ through its entries for $i, j = 1, \ldots, r$, as follows $P(B)_{i j} = 0$ if $i > j$, $P(B)_{i j} = \frac{1}{2} B_{ij}$ if $i = j$, $P(B)_{i j} =  B_{ij}$ if $i < j$.
		Recall that in \cite[Theorem 2.1 and (25)]{CholeskyPerturb} authors prove that
		\begin{equation}\label{eq_stab_chol_ww}
			\| {\rm{Id}} - L \| \leq 2 \| P(B) \|, \quad \text{if} \quad 2 (1 + 2 \lceil \log_2 \dim V \rceil) \| P(B) \| < 1.
		\end{equation}
		Moreover, in \cite[(25)]{CholeskyPerturb}, it is established that 
		\begin{equation}\label{eq_stab_chol_ww121}
			\| P(B) \| \leq (1/2 + \lceil \log_2 \dim V \rceil) \| B \|.
		\end{equation}
		Now, (\ref{eqthm_cholesky}) can be restated as $\| B \| \leq C$, where $\| \cdot \|$ is the operator norm subordinate with $H_0$.
		This with (\ref{eq_stab_chol_ww121}) implies that $2 (1 + 2 \lceil \log_2 \dim V \rceil) \| P(B) \| < 1$ under the stated assumption on $C$.
		The result now follows directly from (\ref{eq_weight_chol}), (\ref{eq_stab_chol_ww}) and the trivial bound $\| A(\mathcal{F}, H_0) \| \leq \| \mathcal{F} \|$.
	\end{proof}
	\par 
	\begin{proof}[Proof of Theorem \ref{thm_subadd_berg}]
		Recall that in (\ref{eq_bergm_kern_ident}), we defined the peak sections, $s_{k, x} \in H^0(X, L^{\otimes k})$ up to a multiplication by a unimodular constant.
		We fix this constant in a compatible way, i.e. so that for any $x \in X$, there is $l \in L_x$, so that for any $k \in \nat$, we have $s_{k, x}(x) = c_k l^{\otimes k}$, for some $c_k > 0$.
		By (\ref{eq_thm_tian}) and (\ref{eq_bergm_kern_ident}), it suffices to establish that there are $k_0 \in \nat$, $C > 0$, so that for any $k, l \geq k_0$, $l / 2 \leq k \leq 2 l$, $x \in X$, we have
		\begin{multline}\label{eq_subadd_weight}
			\big\langle A(\mathcal{F}_{k + l}, {\textrm{Hilb}}_{k + l}(h^L)) s_{x, k + l}, s_{x, k + l} \big\rangle_{{\textrm{Hilb}}_{k + l}(h^L)}
			\geq
			\big\langle A(\mathcal{F}_{k}, {\textrm{Hilb}}_{k}(h^L)) s_{x, k}, s_{x, k} \big\rangle_{{\textrm{Hilb}}_{k}(h^L)}
			\\
			+
			\big\langle A(\mathcal{F}_{l}, {\textrm{Hilb}}_{l}(h^L)) s_{x, l}, s_{x, l} \big\rangle_{{\textrm{Hilb}}_{l}(h^L)}
			-
			C \log(k + l).
		\end{multline}
		From now on, we concentrate on the proof of (\ref{eq_subadd_weight}).
		\par 
		We denote by ${\rm{Mult}}_{k, l}^*$ the dual of the map (\ref{eq_mult_map}), for which the domain is endowed with the norm ${\rm{Hilb}}_k(h^L) \otimes {\rm{Hilb}}_l(h^L)$ and the codomain is endowed with the quotient norm $[{\rm{Hilb}}_k(h^L) \otimes {\rm{Hilb}}_l(h^L)]$.
		We claim that 
		\begin{equation}\label{eq_mult_dual}
			{\rm{Mult}}_{k, l}^* (s_{x, k} \cdot s_{x, l})
			=
			s_{x, k} \otimes s_{x, l}.
		\end{equation}
		To see this, we denote by $\mathcal{J}_x \subset \mathscr{O}_X$ the ideal sheaf of germs of holomorphic functions vanishing at $x$.
		It is immediate to see that ${\rm{Mult}}_{k, l}^* (s_{x, k} \cdot s_{x, l})$ has to be the element $h \in H^0(X, L^{\otimes k}) \otimes H^0(X, L^{\otimes l})$ minimizing the norm ${\rm{Hilb}}_k(h^L) \otimes {\rm{Hilb}}_l(h^L)$, and such that ${\rm{Mult}}_{k, l}(h) = s_{x, k} \cdot s_{x, l}$.
		We express $h$ in the basis $s_{x, k} \otimes s_{x, l}$, $s_{x, k} \otimes s_{x, j, l}$, $s_{x, i, k} \otimes s_{x, l}$, $s_{x, i, k} \otimes s_{x, j, l}$, where $s_{x, i, k}$ (resp. $s_{x, j, l}$), $i = 2, \ldots, N_k$, (resp. $j = 2, \ldots, N_l$) form the bases for $H^0(X, L^{\otimes k} \otimes \mathcal{J}_x)$ (resp. $H^0(X, L^{\otimes l} \otimes \mathcal{J}_x)$).
		By comparing the value at $x$, we see that the coefficient of $s_{x, k} \otimes s_{x, l}$ has to be equal to $1$. 
		Also, due to minimality of the norm, all the other components have to vanish.
		This establishes (\ref{eq_mult_dual}).
		\par 
		We also claim that there are $k_0 \in \nat$, $C > 0$, so that for any $k, l \geq k_0$, $x \in X$, we have
		\begin{equation}\label{eq_peak_sect_prod_str}
			\Big\| 
				s_{x, k} \cdot s_{x, l}
				-
				\Big( \frac{k \cdot l}{k + l} \Big)^{\frac{n}{2}} s_{x, k + l}
			\Big\|_{{\textrm{Hilb}}_{k + l}(h^L)}
			\leq
			C
			\Big( \frac{k \cdot l}{k + l} \Big)^{\frac{n}{2}} \cdot \Big(
				\frac{1}{k} 
				+
				\frac{1}{l}
			\Big).
		\end{equation}
		To see this, remark that $s_{x, k}(y) \cdot \sqrt{B_k(x)} = B_k(y, x) \cdot l^{\otimes k}$, where the dot stands for the pairing between $L_x^*$ and $L_x$.
		Taking this into account, (\ref{eq_peak_sect_prod_str}) follows directly from the estimates of Dai-Liu-Ma \cite{DaiLiuMa} and Ma-Marinescu \cite{MaMarOffDiag}, cf. (\ref{eq_dai_liu_ma}), using the bound (\ref{eq_int_gauss}).
		\par 
		From Proposition \ref{prop_stein_weiss}, we see that for $k_0 \in \nat$, described after (\ref{eq_mult_map}), for any $k, l \geq k_0$, we have
		\begin{multline}\label{eq_subadd_weight1}
			A([\mathcal{F}_k \otimes \mathcal{F}_l], [{\rm{Hilb}}_k(h^L) \otimes {\rm{Hilb}}_l(h^L)])
			\\
			\geq
			\Big(
			A(\mathcal{F}_k, {\rm{Hilb}}_k(h^L)) \otimes {\rm{Id}}
			+
			{\rm{Id}} \otimes A(\mathcal{F}_l, {\rm{Hilb}}_l(h^L))
			\Big)\Big|_{H^0(X, L^{\otimes (k + l)})}.
		\end{multline}
		We also remark that by the definition of submultiplicativity and Proposition \ref{thm_interpol}, we have
		\begin{equation}\label{eq_subadd_weight11}
			A(\mathcal{F}_{k + l}, [{\rm{Hilb}}_k(h^L) \otimes {\rm{Hilb}}_l(h^L)])
			\geq
			A([\mathcal{F}_k \otimes \mathcal{F}_l], [{\rm{Hilb}}_k(h^L) \otimes {\rm{Hilb}}_l(h^L)]).
		\end{equation}
		If we evaluate (\ref{eq_subadd_weight1}) at $s_{x, k} \cdot s_{x, l}$ and then use (\ref{eq_mult_dual}) along with (\ref{eq_subadd_weight11}), and the fact that both $s_{x, k}$ and $s_{x, l}$ are of unit norm, we get 
		\begin{multline}\label{eq_subadd_weight2}
			\Big\langle A(\mathcal{F}_{k + l}, [{\rm{Hilb}}_k(h^L) \otimes {\rm{Hilb}}_l(h^L)]) (s_{x, k} \cdot s_{x, l}), (s_{x, k} \cdot s_{x, l}) \Big\rangle_{[{\rm{Hilb}}_k(h^L) \otimes {\rm{Hilb}}_l(h^L)]}
			\\
			\geq
			\Big\langle
			A(\mathcal{F}_k, {\rm{Hilb}}_k(h^L)) s_{x, k}, s_{x, k}
			\Big\rangle_{{\rm{Hilb}}_k(h^L)}
			+
			\Big\langle
			A(\mathcal{F}_l, {\rm{Hilb}}_l(h^L)) s_{x, l}, s_{x, l}
			\Big\rangle_{{\rm{Hilb}}_l(h^L)}.
		\end{multline}
		\par 
		Remark also that in the notations of Theorem \ref{thm_cholesky}, for an arbitrary $C > 0$, we have $A(\mathcal{F}, H_0) = A(\mathcal{F}, C \cdot H_0)$.
		Also, by the asymptotic Riemann-Roch-Hirzebruch theorem, for any $\epsilon > 0$, there is $k_2 \in \nat$, such that for any $k, l \geq k_2$, $l/2 \leq k \leq 2l$, we have $2 ( 1 + \log N_{k + l} ) \big(\frac{1}{k} + \frac{1}{l}) \leq \epsilon$.
		In particular, by Theorem \ref{thm_as_isom}, the result of Theorem \ref{thm_cholesky} applies to $H_0 := [{\rm{Hilb}}_k(h^L) \otimes {\rm{Hilb}}_l(h^L)] \cdot (\frac{k \cdot l}{k + l})^{n / 2}$ and $H_1 := {\rm{Hilb}}_{k + l}(h^L)$ for  $k, l \geq k_2$, $l/2 \leq k \leq 2l$.
		As a conclusion, there are $C > 0$, $k_3 \in \nat$, such that for any $k, l \geq k_3$, we have
		\begin{equation}\label{eq_subadd_weight3}
			\Big\| A(\mathcal{F}_{k + l}, [{\rm{Hilb}}_k(h^L) \otimes {\rm{Hilb}}_l(h^L)]) 
			-
			A(\mathcal{F}_{k + l}, {\rm{Hilb}}_{k + l}(h^L))
			\Big\|
			\leq
			C \log(k + l),
		\end{equation}
		where $\| \cdot \|$ is for the operator norm either with respect to $[{\rm{Hilb}}_k(h^L) \otimes {\rm{Hilb}}_l(h^L)]$ or ${\rm{Hilb}}_{k + l}(h^L)$.
		\par 
		Directly from Theorem \ref{thm_as_isom}, (\ref{eq_subadd_weight2}) and (\ref{eq_subadd_weight3}), we deduce that there are $k_4 \in \nat$, $C > 0$, so that for any $k, l \geq k_4$, $l / 2 \leq k \leq 2 l$, $x \in X$, we have
		\begin{multline}\label{eq_subadd_weight4}
			\Big\langle A(\mathcal{F}_{k + l}, {\rm{Hilb}}_{k + l}(h^L)) (s_{x, k} \cdot s_{x, l}), (s_{x, k} \cdot s_{x, l}) \Big\rangle_{{\rm{Hilb}}_{k + l}(h^L)}
			\cdot
			\Big( \frac{k + l}{k \cdot l} \Big)^{n}
			\\
			\geq
			\Big\langle
			A(\mathcal{F}_k, {\rm{Hilb}}_k(h^L)) s_{x, k}, s_{x, k}
			\Big\rangle_{{\rm{Hilb}}_k(h^L)}
			\\
			+
			\Big\langle
			A(\mathcal{F}_l, {\rm{Hilb}}_l(h^L)) s_{x, l}, s_{x, l}
			\Big\rangle_{{\rm{Hilb}}_l(h^L)}
			-
			C \log(k + l).
		\end{multline}	 
		The statement (\ref{eq_subadd_weight}) now follows directly from (\ref{eq_peak_sect_prod_str}) and (\ref{eq_subadd_weight4}).
	\end{proof}

\bibliography{bibliography}

		\bibliographystyle{abbrv}

\Addresses

\end{document}